\documentclass[graybox]{svmult}

\usepackage{mathptmx}       
\usepackage{helvet}         
\usepackage{courier}        
\usepackage{type1cm}        
\usepackage{makeidx}         
\usepackage{graphicx}        
\usepackage{multicol}        
\usepackage[bottom]{footmisc}

\usepackage{amsmath}
\usepackage[UKenglish]{babel}
\usepackage{amssymb}
\usepackage{caption}

\makeindex             

\usepackage{amsfonts}
\usepackage{amsmath}
\usepackage[UKenglish]{babel}
\usepackage{graphicx}
\usepackage{amssymb}

\definecolor{RoyalBlue}{cmyk}{1, 0.50, 0, 0}
\definecolor{Coral}{RGB}{240,90,110}
\definecolor{Maroon}{cmyk}{0, 0.87, 0.68, 0.32}

\begin{document}

\title*{Tensor completion in hierarchical tensor representations}
\author{Holger Rauhut, Reinhold Schneider and \v{Z}eljka Stojanac }
\institute{Holger Rauhut \at RWTH Aachen University, Lehrstuhl C f{\"u}r Mathematik (Analysis), Templergraben 55, 52062 Aachen Germany, \email{rauhut@mathc.rwth-aachen.de}
\and Reinhold Schneider \at Technische Universit{\"a}t Berlin, Stra\ss e des 17. Juni 136, 10623 Berlin,  \email{schneidr@math.tu-berlin.de}
\and \v{Z}eljka Stojanac \at RWTH Aachen University, Lehrstuhl C f{\"u}r Mathematik (Analysis), Templergraben 55, 52062 Aachen Germany, \email{stojanac@mathc.rwth-aachen.de}}
%
%
\maketitle

\abstract{Compressed sensing extends from the recovery of sparse vectors from undersampled measurements via efficient algorithms
to the recovery of matrices of low rank from incomplete information. Here we consider a further extension
to the reconstruction of tensors of low multi-linear rank in recently introduced hierarchical tensor formats
from a small number of measurements. Hierarchical tensors are a flexible generalization of the well-known Tucker representation, which have
the advantage that the number of degrees of freedom of a low rank tensor 
does not scale exponentially with the order of the tensor. While corresponding tensor decompositions can be computed efficiently via
successive applications of (matrix) singular value decompositions, some important properties of the singular value decomposition
do not extend from the matrix to the tensor case. This results in major computational and theoretical difficulties in designing and analyzing
algorithms for low rank tensor recovery. For instance, a canonical analogue of the tensor nuclear norm is NP-hard to compute
in general, which is in stark contrast to the matrix case. In this book chapter we consider versions of iterative hard thresholding schemes
adapted to hierarchical tensor formats. A variant builds on methods from Riemannian optimization and 
uses a retraction mapping from the tangent space of the manifold of low rank tensors back to this manifold.
We provide first partial convergence results based on a tensor version of the restricted isometry property (TRIP)
of the measurement map. Moreover, an estimate of the number of measurements is provided
that ensures the TRIP of a given tensor rank with high probability for Gaussian  
measurement maps.
}

\section{Introduction} 

As outlined in the introductory chapter of this book, {\it compressed sensing} allows the recovery of (approximately) sparse vectors from
a small number of linear random measurements via efficient algorithms including $\ell_1$-minimization and iterative hard thresholding, see also
\cite{rauhut,elku12} for introductory material. This theory was later extended to the reconstruction of low rank matrices 
from random measurements in \cite{brecht,cata10,capl11,gr11}. 
An important special case includes the {\it matrix completion problem}, where one seeks to fill in missing 
entries of a low rank matrix \cite{care09,cata10,gr11,re12}. Corresponding algorithms include {\it nuclear norm minimization} \cite{brecht,cata10,fa02-2} and versions of {\it iterative hard thresholding} \cite{tanner}.

In the present article we pursue a further extension of compressed sensing. We consider
the recovery of a low rank tensor  
from a relatively  small number of measurements. In contrast to already existing work in this direction \cite{recht,goldfarb},
we will understand low rank tensors in the framework of
recently introduced  {\em hierarchical tensor formats} \cite{hackbusch}. 
This concept includes the classical  
{\em Tucker format}  \cite{tucker,kolda,hosvd} as well as {\em tensor trains} \cite{oseledets,oseledets1}. 
These hierarchical tensors can be represented in a data sparse way,  i.e., they require only a very low number 
of data for their representation compared to the dimension of the full tensor space. 

Let us recall the setup of low rank matrix recovery first.
Given a matrix $\mathbf{X} \in \mathbf{R}^{n_1 \times n_2}$
of rank at most $ r \ll  \mbox{ min } \{n_1,n_2 \}$, the goal
of  low rank matrix recovery  is to reconstruct 
$\mathbf{X}$ from linear measurements  $b_i =1, \ldots, m$, i.e., $\mathbf{b} = \mathcal{A} (\mathbf{ X})$, 
where $\mathcal{A}  : \mathbb{R}^{n_1\times n_2} \to 
 \mathbb{R}^m$  with $m \ll n_1 n_2$  is a linear sensing operator.

This problem setting can be  
transferred to the problem 
to  recover higher order tensors 
$\mathbf{ u} \in  \mathcal{H}_d := \mathbb{R}^{n_1 \times n_2 \times \cdots  \times n_d}$, 
$ \pmb{\mu} = (\mu_1,\ldots , \mu_d) \mapsto \mathbf{ u} (\mu_1,\ldots , \mu_d) $  from 
the linear measurements 
$\mathbf{b} = \mathcal{A} (\mathbf{u} )$, where 
$\mathcal{A} : \mathbb{R}^{n_1 \times n_2 \times \cdots \times n_d} \to \mathbb{R}^m$, is a sensing operator
$ m \ll n_1 n_2 \cdots  n_d$.
Here $ d$ denotes the order of the tensor (number of modes of a tensor) and we remark that, for easier readability, we use the notation $\mathbf{ u} (\mu_1,\ldots , \mu_d)$ referring
to the entries of the tensor.
In the present article we assume that the tensors to be reconstructed belong
to a class of hierarchical tensors of  a given  low multi-linear rank $\mathbf{ r} = (r_j)_{j=1}^p$, where $p$ depends on the specific tensor format \cite{dasilva}, see also below.
Of particular interest is the special case of {\em tensor completion},  where the measurement operator
samples entries of the tensor, i.e.,
$$ \big(
 \mathcal{A} 
  \mathbf{u} \big)_i
  =
 \mathbf{u} ({\pmb{\mu}_i })  = \mathbf{ u}  (\mu_{1,i} , \ldots , \mu_{d,i} ) = b_i \ \ , \ \ i=1, \ldots , m  \ ,  $$
where the $ \pmb{\mu}_i   \in \Omega$, $|\Omega| = m$,  are given 
(multi-)indices 
\cite{mohlenkamp,signoretto1,kressner}.
Tensors, even of high order $ d \gg 3$,  appear frequently in data and signal  analysis. 
 For example, a   video signal 
is a tensor of order $d=3$.   High order tensors of relatively low rank 
 may also arise in  vector-tensorization \cite{hackbusch-tensorisation,oseledets-qtt}
of a low dimensional signal. 
The present article tries to present a framework for  tensor recovery 
 from  the perspective of recent developments in  tensor product approximation
\cite{kolda,belmol,hackbusch},  in particular the development of hierarchical tensors \cite{hackbusch,oseledets}. 
The  {\em canonical format} (CANDECOMP, PARAFAC)  representing a tensor of order $d$ as a sum of 
elementary tensor products, or rank one tensors (see \cite{kolda,belmol})
\begin{eqnarray}
\mathbf{u} ( \mu_1 , \ldots , \mu_d ) & = & \sum_{k=1}^r 
\big( \mathbf{c}^1_k \otimes \cdots \otimes 
\mathbf{c}^d_k \big) ( \mu_1 , \ldots , \mu_d )\label{canonical:format} \\
&=& \sum_{k=1}^r \mathbf{c}^1_k  (\mu_1)  \cdots \mathbf{c}^d_k (\mu_d ) \ \ ,
\ \ \mu_i = 1, \ldots , n_i \ , \ i=1 , \ldots , d \ , \notag
\end{eqnarray}
with $ \mathbf{c}^i_k \in \mathbb{R}^{n_i}$, 
suffers from severe difficulties, unless $ d\leq 2$. 
For example, the tensor rank is not well defined,  and the set of tensors of the above form with fixed 
$r$ is not closed \cite{lim} and does not form an algebraic variety. 
However, we obtain a closed subset, if we impose further conditions. Typical examples for such
conditions are e.g.\ symmetry \cite{landsberg}  
or bounds $\left|\mathbf{c}_k^i\right| \leq \alpha$ for some fixed $\alpha$ \cite{usch}.

However, it has been experienced that computations within 
the {\em Tucker  tensor format}  behave relatively robust and stable, whereas the complexity unfortunately 
still suffers from the curse of dimensionality. 
A  first important  observation may be summarized in the 
fact that the set of Tucker tensors with a Tucker rank at most  $ \mathbf{r} =(r_1, \ldots , r_d)$ forms an algebraic variety, i.e., 
a set of common zeros of multi-variate polynomials.  
Recently developed {\em hierarchical tensors}, introduced by Hackbusch and coworkers (HT tensors) 
\cite{hackbusch-kuehn,lars} and
the group of Tyrtyshnikov (tensor trains, TT) \cite{oseledets,oseledets1} 
have extended the Tucker   
format \cite{tucker,kolda,hosvd} into a multi-level framework, that no longer suffers from high order scaling w.r.t.\ the order $d$, 
as long as the ranks are moderate. For $d=3$ there is no essential difference, whereas for larger $ d \geq 4$ one benefits from the use of the novel formats. 
This makes the Tucker format  \cite{signoretto1,goldfarb,yin} and, in  particular,  its hierarchical  generalization
\cite{dasilva,kressner,rss1},  the {\em hierarchical  tensor format}, a proper candidate for 
tensor product approximation in the sense that it serves as 
an appropriate model class in which  we would like to represent or approximate tensors of interest in a data sparse way. 
Several algorithms developed in compressed sensing and matrix recovery or matrix completion 
can be easily transferred to this tensor setting (with the exception of nuclear norm minimization, which poses some fundamental difficulties).
However, we already note at this point that the analysis of algorithms is much harder for tensors than for matrices as we will see below.

Historically, the hierarchical tensor framework 
has  evolved  in the quantum physics community 
hidden in the renormalization group ideas \cite{white}, 
and became clearly visible  in the framework of matrix product and tensor network states
\cite{schollwoeck}.  An
independent   source of these developments can be found   in quantum dynamics as the 
multi-layer {multiconfigurational time-dependent Hartree} {(MCTDH)} method \cite{mayer,wang,lubich-blau}. 
Only after the  recent  introduction  of hierarchical tensor representations  in numerics, 
namely  {H}ierarchical Tucker (HT) \cite{hackbusch,hackbusch-kuehn} 
and  {T}ensor {T}rains (TT) \cite{oseledets,oseledets1}, its relationship to already  existing 
concepts in quantum physics 
has  been  realized \cite{lrss}. 
We refer the interested reader to {the} recent survey articles \cite{hackbuschAN,grasedyck-kressner,lrss,hs} 
and the monograph \cite{hackbusch}. 
In the present paper we would like to provide a fairly self-contained introduction, and demonstrate how these concepts
can be applied for tensor recovery. 

There are several essential difficulties when passing from matrix to tensor recovery.
In matrix recovery, the original problem can be reformulated to finding the solution of the optimization problem
$$ 
\mbox{minimize  rank} ({\bf Z}) \mbox{ s.t. } \mathcal{A} ({\bf Z} ) = {\bf b}  \ , \ 
 {\bf Z} \in \mathbb{R}^{n_1 \times n_2},  
$$
i.e., to finding the matrix $ \mathbf{Z}$ with the lowest rank consistent with the measurements.
While this problem is NP-hard \cite{fa02-2}, 
it can be relaxed to the convex optimization problem of constrained nuclear norm minimization 
$$\mbox{ minimize  } \| {\bf Z} \|_{*}  \mbox{  s.t. } \mathcal{ A } ({\bf Z}) ={\bf b}   \ , \ 
 {\bf Z} \in \mathbb{R}^{n_1 \times n_2}  \ .$$
Here, the nuclear norm is the sum of the singular values $\sigma_j({\bf Z})$ of ${\bf Z}$, i.e., $\|{\bf Z}\|_* = \sum_j \sigma_j({\bf Z})$.
The minimizer of this problem reconstructs ${\bf X}$ exactly under suitable conditions on $\mathcal{A}$ \cite{brecht,capl11,care09,gr11,rauhut}.

In the hierarchical tensor setting, we are dealing with 
a rank tuple $\mathbf{r} = (r_1, \ldots , r_p)$, which we would like to minimize simultaneously. 
However, this is not the only difficulty arising from the non-linear 
tensor ansatz. In fact, the tensor nuclear norm is NP-hard to compute \cite{NPcomplete, Fried} and therefore,
tensor nuclear norm minimization is computationally prohibitive. 
Another difficulty arises because in contrast to the matrix case, also the best rank ${\mathbf r}$-approximation
to a given tensor is NP-hard to compute \cite{NPcomplete, Fried}.

Our model class of hierarchical tensors of fixed multi-linear rank $\mathbf{r}$
is a smooth embedded manifold, and its closure constitutes an algebraic variety.   
These are properties on which one can built local optimization methods \cite{absil,lrsv,su2}, 
subsumed under the moniker {\it Riemannian optimization}.
Moreover, for {\em hierarchical tensor representation}  
efficient numerical tools for finding at least a quasi-best approximation are available, namely
the  {\em higher order singular value decomposition (HOSVD)}, related to the Tucker model \cite{hosvd}, or the {\em hierarchical  singular value decomposition (HSVD)}, which is an extension of the HOSVD to hierarchical 
Tucker models \cite{oseledets2,hackbusch,lars,vidal}. All these methods proceed 
via successive computations of the SVD of certain matricisations of the original tensor. 

The HSVD (and the HOSVD as a special case) enables us to compute rank ${\mathbf{r}}$ approximations to a given
tensor via truncation of the decomposition. This allows to extend a particular class of  
greedy type algorithms, namely {\em iterative hard 
thresholding algorithms} to the present tensor setting. 
In a wider sense, this class of techniques includes also 
related Riemannian manifold techniques \cite{kressner,dasilva}  and 
alternating least squares methods \cite{hrs-als}. 
First numerical tests show promising results \cite{dasilva,kressner,rss1,rss2}. 
For a convergence analysis in the tensor case, and in applications, 
however, we have to struggle with more and harder difficulties than in the matrix case.
The most fundamental of these consists in the fact, that truncations of the HSVD only provide
quasi-best low rank approximations. Although bounds of the approximation error are known \cite{lars},
they are not good enough for our purposes, which is the main reason why we are only able to provide
partial convergence results in this chapter.
Another difficulty with iterative hard thresholding  algorithms is that the
rank, here a rank tuple $\mathbf{r} = ( r_1 , \ldots , r_p)$, has to be fixed a priori. 
In practice this rank tuple is not known in
advance, and a strategy for specifying appropriate ranks is required. 
Well known strategies borrowed from matrix recovery consist in increasing the rank during the approximation or starting with overestimating 
the rank and reduce the ranks through the iteration \cite{wutao}. 
For our seminal treatment, we simply assume that  the multi-linear rank $\mathbf{r} $  
is known in advance, i.e.,  the sought tensor $\mathbf{u}$  is of exact rank $\mathbf{r}$. 
Moreover, we assume noiseless measurements  
$ \big( \mathcal{A} \mathbf{u}\big)_j = b_j \in \mathbb{R}$, $j =1, \ldots , m$. 
The important issues of adapting ranks and obtaining robustness will be deferred to future research \cite{rss2}. 


\medskip

Our chapter is related to the one on {\em two algorithms for compressed sensing of sparse tensors} by 
S.~Friedland, Q.~Li, D.~Schonfeld and E.E.~Bernal. The latter chapter also considers the recovery of mode $d$-tensors from incomplete information using
efficient algorithms. However, in contrast to our chapter, the authors assume usual sparsity of the tensor instead of the tensor being of low rank. The tensor structure
is used in order to simplify the measurement process and to speed up the reconstruction rather than to work with the smallest possible number
of measurements and to exploit low-rankness.

\section{Hierarchical Tensors} 

\subsection{Tensor product spaces} 
We start with some preliminaries. 
In the sequel, we consider only the real field $\mathbb{K} = \mathbb{R}$, 
but most parts are easy to extend to the complex case  
as well.  We will  confine ourselves to finite dimensional linear spaces $ V_i = \mathbb{R}^{n_i} $ from which the  tensor product space 
$$ \mathcal{H}_d = \bigotimes_{i=1}^d V_i := \bigotimes_{i=1}^d \mathbb{R}^{n_i}  \ \  , $$
is built \cite{hackbusch}. 
If it is not stated explicitly, the 
$V_i = \mathbb{R}^{n_i} $ are supplied with  the canonical basis $ \{  \mathbf{e}^i_1 , \ldots, \mathbf{e}^i_{n_i}   \} $ of the vector space
$\mathbb{R}^{n_i} $. Then any $ \mathbf{u} \in \mathcal{H}_d$ can be represented as
\begin{eqnarray*}
 \mathbf{u}   & = & \sum_{\mu_1=1 }^{n_1}  \ldots \sum_{\mu_d = 1}^{n_d}  \mathbf{u}  (\mu_1 , \ldots , \mu_d ) \;
 \mathbf{e}^1_{\mu_1}  \otimes \cdots \otimes   \mathbf{e}^d_{\mu_d}   \ .
\end{eqnarray*}
Using this basis, with a slight abuse of notation,  we can  
identify $   \mathbf{u}   \in \mathcal{H}_d  $ with its representation by 
 a d-variate function, often called hyper matrix, 
$$ \pmb{\mu} = 
(\mu_1, \ldots , \mu_d ) \mapsto \mathbf{u}  (\mu_1, \ldots , \mu_d )  \in \mathbb{R} \ \ , \ \ 
\mu_i = 1, \ldots, n_i  \ , \ i=1, \ldots , d \ ,
$$
depending on discrete variables, usually called indices $\mu_i = 1, \ldots , n_i $, and $ \pmb{\mu} $ 
is called a multi-index. 
Of course, the actual representation  $\mathbf{u} ( \cdots ) $  of $\mathbf{u} \in \mathcal{H}_d$
depends on the chosen bases of $V_1, \ldots , V_d$. 
With $n= \max \{n_i : i=1, \ldots, d \}$, the number of possibly non-zero entries in the representation of $\mathbf{u} $ is $ n_1 \cdots n_d = \mathcal{O} (n^d) $.
This is often referred to as the {\em curse of dimensions}. 
We equip the linear space 
$\mathcal{H}_d $ with the   $\ell_2 $-norm $ \|\mathbf{u} \| = \sqrt{\langle  \mathbf{u} ,  \mathbf{u} \rangle}$ and {the} inner product 
$$
\langle  \mathbf{u}  ,  \mathbf{v}  \rangle := \sum_{\mu_1= 1}^{n_1}  \cdots \sum_{\mu_d =1}^{n_d}
  { \mathbf{u} (\mu_1, \ldots , \mu_d )}
 \mathbf{v} (\mu_1, \ldots , \mu_d )  \ . 
$$
We distinguish linear operators between vector spaces and their corresponding 
 representation by matrices, which are written  by capital bold letters $ \mathbf{U}$. 
Throughout this chapter, all tensor contractions or various tensor--tensor  products are either 
defined explicitly, by summation over corresponding indices,  or 
by introducing corresponding matricisations of the tensors and performing matrix--matrix products.

\subsection{Subspace approximation}

The essence of the classical Tucker format is that, given a tensor $\mathbf{u}$  and a rank-tuple $\mathbf{r} = (r_j)_{j=1}^d$,
one is searching  
for optimal subspaces $U_i \subset \mathbb{R}^{n_i}$ such that
\[
\min \| \mathbf{u }  - \mathbf{v}\|, \mbox{ where } \mathbf{v} \in U_1 \otimes \cdots \otimes U_d ,
\]
is minimized over $U_1,\hdots,U_d$ with $\dim U_i = r_i$. 
Equivalently, we are looking for corresponding  bases $\mathbf{b}^i_{k_i}$ of $U_i$, which can be written in the form 
 \begin{equation} \label{leaftransfer}
  \mathbf{b}^i_{k_i} : =  \sum_{\mu_i =1}^{n_i}   \mathbf{b}^i  ( \mu_i, k_i ) \mathbf{e}^i_{\mu_i}  \ \ , \ \ 
 \ k_i = 1 ,\ldots , r_i  < n_i \ , 
\end{equation}
where $\mathbf{b}^i  (k_i , \mu_i ) \in \mathbb{R}$, 
for each  coordinate direction  
 $i=1,\ldots,d$. With a slight abuse of notation  we often   identify  the basis vector with its representation
 $$
 \mathbf{b}^i_{k_i} \simeq \big(    \mu_i \mapsto  \mathbf{b}^i  ( \mu_i, k_i )  \big) �\ , \ \mu_i = 1, \ldots, n_i \ , \ 
  \ k_i = 1 ,\ldots , r_i \ ,
 $$
 i.e.,  a discrete function or  an  $n_i$-tuple. 
This concept of subspace approximation can be used either  for an  approximation 
$\mathbf{u}$ of a single  tensor,  as well as for  an ensemble of  
tensors $ \mathbf{u}_{j}   $, $j  =1, \dots, m$, in tensor product spaces. 
Given the bases $ \mathbf{b}^i_{k_i} $, 
$ \mathbf{u}_j $ can be represented by  
\begin{eqnarray}
\mathbf{u}_j   = \sum_{k_1 = 1}^{r_1}  \ldots \sum_{k_d =1}^{r_d}    \mathbf{c}   (j, k_1, \ldots , k_d  )
  {\bf b}^1_{k_1}  \otimes \cdots  \otimes  {\bf b}^d_{k_d}   \in
 \bigotimes_{i=1}^d  U_i \subset   \mathcal{H}_d = \bigotimes_{i=1}^d {\mathbb{R}}^{n_i}  \ . \label{eq:tucker}
\end{eqnarray}
In case $\mathbf{b}^i$'s form orthonormal bases, the core tensor $\mathbf{c}  \in \mathbb{R}^m \otimes \bigotimes_{i=1}^d \mathbb{R}^{r_i} $ is given entry-wise by 
\begin{equation*} 
\mathbf{c} (j, k_1, \ldots , k_d  )  = \langle \mathbf{u}_j  ,  \mathbf{b}^1_{k_1}  \otimes \cdots \otimes �\mathbf{b}^d_{k_d} \rangle   
   \ . 
\end{equation*}  
We call a representation of the form (\ref{eq:tucker}) with some $ \mathbf{b}_{k_i}^i , \mathbf{c} $ 
a Tucker representation, and the  Tucker representations the Tucker format. In this formal parametrization, the upper limit of the sums may be larger than the ranks 
and $\{\mathbf{b}_{k_i}^i\}_{k_i}$ may not be linearly independent. Noticing that a Tucker representation of a 
tensor is not uniquely defined, we are interested in some normal form.

Since the core tensor contains $r_1 \cdots r_d \sim r^d $, $r := \max \{ r_i : i=1 , \ldots, d \}$,
possibly nonzero entries, this concept does not prevent the number of free parameters
from scaling exponentially with the dimensions $ \mathcal{O} (r^d)$. Setting  $n := \max \{ n_i : i=1, \ldots , d\}$, 
the overall complexity for storing the required data (including the basis vectors) is bounded by $\mathcal{O} ( ndr + r^d) $.
Since  $n_i$ is replaced by $r_i$, one obtains a dramatical compression 
$ \frac{r_1 }{n_1}  \cdots \frac{r_d }{n_d} \sim  \big( \frac{r }{n} \big)^d$.  
Without further sparsity of the core tensors the Tucker format is appropriate for low order tensors
$ d < 4$.

\subsection{Hierarchical tensor representation}  

The \emph{hierarchical Tucker format} (HT)  in the form introduced by Hackbusch and K\"uhn in \cite{hackbusch-kuehn},  
extends the idea of subspace approximation to a hierarchical or multi-level framework. 
Let us  proceed   in a hierarchical way. 
We first consider 
$ V_1 \otimes V_2 = \mathbb{R}^{n_1} \otimes \mathbb{R}^{n_2}$ or preferably the subspaces 
$U_1 \otimes U_2$ introduced in the previous section.
For the approximation of $\mathbf{u} \in \mathcal{H}_d$ we 
only need a subspace $U_{\{ 1,2 \} } \subset  U_1 \otimes U_2 $ with dimension $r_{\{1,2\}}  < r_1  r_2$. 
Indeed, $V_{\{ 1,2\}} $  is defined through a new basis 
\begin{eqnarray*} 
V_{\{ 1,2\}}  & = & \mbox{span } 
\{  \mathbf{b}^{\{ 1,2 \} }_{k_{ \{ 1,2 \}}}  :  k_{ \{ 1,2\} }  = 1, \ldots , r_{\{ 1,2\} }  \}  \ ,
\end{eqnarray*}
with basis vectors   given by 
$$
  \mathbf{b}^{\{ 1,2 \} }_{k_{ \{ 1,2 \}}}  = \sum_{{k}_1 =1 }^{r_1} \sum_{{k}_2 =1}^{r_2}
  \mathbf{b}^{\{ 1,2 \} } ( k_1 , k_2 ,  k_{\{ 1,2 \} }  )  \;  {\bf b}^1_{k_1}   \otimes  {\bf b}^2_{k_2 } , \quad k_{\{ 1,2 \} }  =1,\hdots,r_{\{ 1,2\} }. 
$$
One may continue in several ways, e.g.~by building a subspace  
$U_{\{ 1,2,3\} } \subset  U_{\{ 1,2\} } \otimes U_3  \subset U_1 \otimes U_2 \otimes U_3 \subset 
V_1 \otimes V_2 \otimes V_3 $, or 
  $U_{\{ 1,2,3,4\} } \subset  U_{\{ 1,2\} } \otimes U_{\{3,4\}}  $, where $U_{\{3,4\}} $ is defined analogously to $U_{\{1,2\}} $
   and so on.  

For a systematic treatment, this approach can be cast into the framework of a partition tree, 
with leaves $ \{1\}, \ldots \{d\} $, 
simply abbreviated here  by $1,\ldots,d$, 
and vertices  $ \alpha  \subset D:= \{ 1, \ldots , d\}$,   
corresponding to the partition $ \alpha = \alpha_1 \cup \alpha_2 $ , $ \alpha_1 \cap \alpha_2  = \emptyset$. Without loss of generality, we can assume that $i <j$, for all $i \in \alpha_1$, $j \in \alpha_2$. We call $\alpha_1, \alpha_2 $ the {\em sons} of the {\em father}  $\alpha$ and
$D$ is called the {\em root} of the tree. 
In the example above we have $ \alpha  : = \{ 1,2,3 \} = \alpha_1 \cup \alpha_2 = \{ 1,2\} \cup \{ 3\}$, 
where $\alpha_1 :=  \{ 1,2 \} $ {and} $ \alpha_2 : = \{ 3 \}  $. 

In general, we do not need to restrict the number of sons, and define the {\em coordination number} 
by the number of sons $+1$ (for the father).
Restricting to a binary tree so that each node contains two sons for non-leaf nodes (i.e. $\alpha \not= \{ i \} $),
 is often the common choice, which we will also consider here.    
Let $ \alpha_1 , \alpha_2  \subset D $ be the two sons of $ \alpha \subset D  $,  
then  $ U_{\alpha } \subset U_{\alpha_1} \otimes U_{\alpha_2} $, is defined by a  basis 
 \begin{equation} \label{eq:transfer} 
  \mathbf{b}^{\alpha}_{\ell}   = \sum_{i=1 }^{r_{\alpha_1} } \sum_{j=1}^{r_{\alpha_2}} 
      \mathbf{b}^{\alpha} (i,j ,  \ell   )  \mathbf{b}^{\alpha_1 }_i  \otimes  \mathbf{b}^{\alpha_2}_j  \ .  
\end{equation}
They can also be considered as matrices 
$ ( \mu_{\alpha}, \ell ) \mapsto  \mathbf{B}^{\alpha} (\mu_{\alpha} , \ell )  \in  \mathbb{R}^{n_{\alpha} \times r_{\alpha}}$ with $n_{\alpha}=\prod_{\ell \in \alpha} n_{\ell}$, for $\alpha \neq \{i\}$.
Without loss of generality, all basis vectors, e.g.~$\{ \mathbf{b}^{\alpha }_{\ell}: \ell=1,\ldots,r_{\alpha}   \} $,
   can be constructed to be orthonormal, as long as $\alpha \not= D$ is not the root.
The tensors $(\ell , i,j) \mapsto \mathbf{b}^{\alpha} (i,j , \ell ) $ will be called {\em transfer} or {\em component tensors}. For a leaf $\{i\} \simeq i$, the tensor $(\mu_i,k_i)\mapsto \mathbf{b}^i(\mu_i,k_i)$ in \eqref{leaftransfer} denotes a {\em transfer} or {\em component} tensor.
The component tensor $ \mathbf{b}^{D} = \mathbf{b}^{ \{1,  \ldots, d\} } $  at the root   is called the {\em root tensor}.

Since the matrices $\mathbf{B}^{\alpha}$ are too large, we avoid computing them. We store
only the {\em transfer} or {\em component tensors} which, for fixed $\ell = 1, \ldots, r_{\alpha}$, can also be casted into
{\em transfer matrices} $ (i,j) \mapsto \mathbf{B}_{\alpha}  (\ell)_{i,j} \in \mathbb{R}^{r_{\alpha_1}  \times r_{\alpha_2} }  $. 

\begin{proposition}[\cite{hackbusch}]
A tensor $ \mathbf{u} \in \mathcal{H}_d $ is completely parametrized by the transfer tensors
$\mathbf{b}^{\alpha} $, $ \alpha \in \mathbb{T} $, i.e., by a multi-linear  function $\tau $ 
\begin{equation*} 
\big(  \mathbf{b}^{\alpha} \big)_{ \alpha \in \mathbb{T}} \mapsto
\mathbf{u} = \tau \big( \{ \mathbf{b}^{\alpha} : \alpha \in \mathbb{T} \}  \big) \ . 
\end{equation*}
\end{proposition}
Indeed $\tau$ is defined  by applying
(\ref{eq:transfer}) recursively. Since $\mathbf{b}^{\alpha} $ depends bi-linearly on  $\mathbf{b}^{\alpha_1} $ and 
$\mathbf{b}^{\alpha_2}$, the composite function $ \tau $ is multi-linear in its arguments $ \mathbf{b}^{\alpha}$. 

\begin{center}
\includegraphics{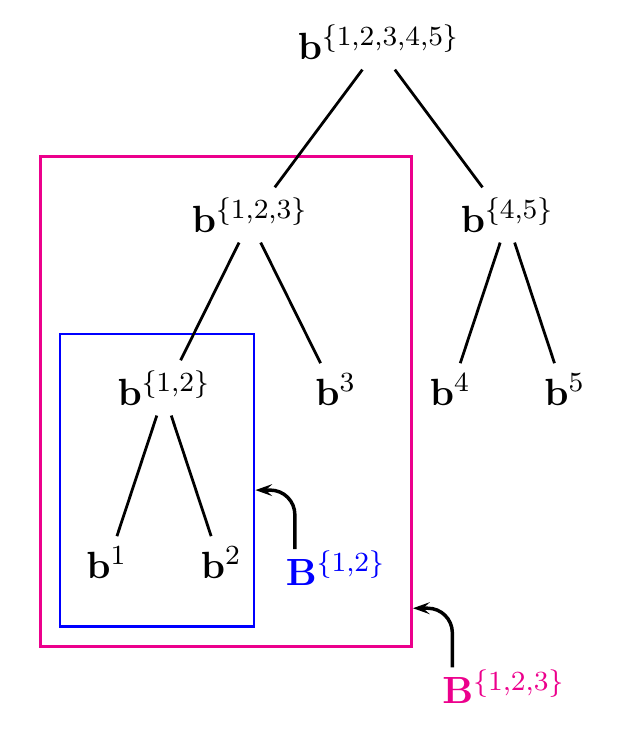} \\
{Hierarchical Tensor representation of an order $5$ tensor
 }
\label{fig:hierarchical}
\end{center}

{\bf Data complexity} Let $ n :=  \max \{n_i:i=1 , \ldots , d\} $, $r := \max \{� r_{\alpha} : \alpha \in \mathbb{T} \}$. 
Then the number of data required for the representation is 
$\mathcal{O} (ndr +  d r^3) $, in particular does not scale exponentially w.r.t. the order $d$.

\subsection{Tensor trains and matrix product representation} 

We now highlight another particular case of hierarchical tensor representations, namely  {\em Tyrtyshnikov tensors (TT)} or {\em tensor trains} 
and {\em  matrix product representations}  
 defined by  taking $ U_{ \{ 1, \ldots , p +1 \}} \subset U_{\{ 1, \ldots , p \} } \otimes V_{\{p+1\}} $,  
 developed as TT tensors (tensor trains) by \cite{oseledets1,oseledets2} 
  and known as  matrix product states (MPS) in physics. 
Therein, we abbreviate $ i \simeq \{ 1, \ldots , i\} $ and consider the unbalanced tree 
$\mathbb{T} = \{  \{ 1\} , \{2\},\{ 1,2\} , \{3\}, \{1,2,3\} ,\ldots ,\{d\},\{ 1, \ldots, d\}\} $ and setting $r_0=r_d=1$.
The  transfer tensor $\mathbf{b}^{\alpha}$ for a leaf $ \alpha \in \{\{2\}, \{3\},\ldots, \{d\}\}$ is usually defined as identity matrix of appropriate size
and therefore the tensor $\mathbf{u} \in \mathcal{H}_d$ is completely parametrized by transfer tensors $\left(\mathbf{b}^{\alpha}\right)_{\alpha \in \mathbb{T}'}$, where $\mathbb{T}'= \{ 1, \ldots , d\}=\{\{1\},\{1,2\},\ldots$, $\{1,2,\ldots,d\}\}$. 
Applying the recursive construction,  the tensor $\mathbf{u}$  can be written as
\begin{eqnarray}
 ( \mu_1, \ldots , \mu_d )  & \mapsto & \mathbf{u} ( \mu_1, \ldots , \mu_d ) \notag  \\ 
 & = & \sum_{k_1=1}^{r_1} \ldots\sum_{k_{d-1}=1}^{r_{d-1}} 
  {\mathbf{b}^1 ( \mu_1 , k_1) \mathbf{b}^2 (k_1,\mu_2 ,k_2) \cdots  
 \mathbf{b}^d (k_{d-1},\mu_d )} \ .  \label{eq:utt} 
\end{eqnarray}
  Introducing   the 
  matrices 
$  \mathbf{B}_i (\mu_i) \in \mathbb{R}^{r_{i-1} \times r_{i} }  $,
$$
\big( \mathbf{B}_i (\mu_i) \big)_{k_{i-1},k_i } = \mathbf{b}^i  ( k_{i-1} , \mu_i , k_i  ) \ , \ 1 \leq i \leq d \ , 
$$
and with the convention 
$r_0 =r_d =1$   
$$
\big( \mathbf{B}_1 (\mu_1) \big)_{k_1 }^* = \mathbf{b}^1 ( \mu_1 , k_1 ) \  \ \text{and} \ \  
\big( \mathbf{B}_d (\mu_d) \big)_{k_{d-1} } = \mathbf{b}^d (k_{d-1}, \mu_d ) \ ,
$$
 the formula (\ref{eq:utt}) can 
 be rewritten entry-wise  by
matrix--matrix products
\begin{equation}
\label{eq:mps} 
  \mathbf{u} (\mu_1, \ldots, \mu_d )  = {\mathbf{B}_{1} (\mu_{1})  
\cdots \mathbf{B}_{i} (\mu_{i}) \cdots \mathbf{B}_{d} (\mu_{d})} = \tau ( \mathbf{b}^1 , \ldots ,  \mathbf{b}^{d} ) 
\ . 
\end{equation}
This representation is by no means unique. In general, there exist  
 $ \mathbf{b}^{\alpha} \neq \mathbf{c}^{\alpha} $ such that $ \tau ( \{  \mathbf{b}^{\alpha}   : \alpha \in \mathbb{T}\}   ) = 
 \tau ( \{  {\mathbf{c}}^{\alpha}   : \alpha \in \mathbb{T}  \} )$. 
 
 \begin{center}
\includegraphics{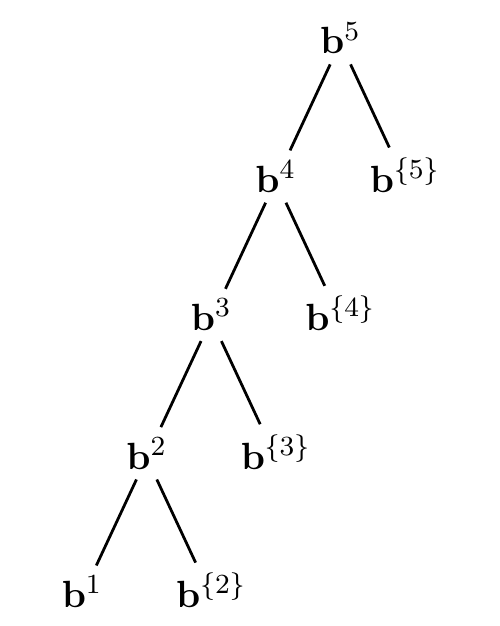} \\
{TT representation of an order $5$ tensor with abbreviation $i \simeq \{1,\ldots,i\}$
 }
\label{fig:tt}
\end{center}

The tree is ordered according to the father-son relation into a hierarchy of levels, where $\mathbf{b}^d$ is the  root tensor. 
Let us observe that we can rearrange the hierarchy in such a way that any node
 $ p =1 , \ldots ,  {d}$ can form  the root of the tree,   i.e., $  \mathbf{b}^{p}  $ becomes  the root tensor. 
Using only orthogonal basis vectors,
which is the preferred choice,  this ordering reflects left and right hand orthogonalization in matrix product states
\cite{hrs-tt}. 
 
A tensor in canonical form 
$$ \mathbf{u} = \sum_{k=1}^{R} \mathbf{u}_k^1 \otimes \cdots \otimes \mathbf{u}_k^d $$
can be easily written in the TT form, by setting $r_i=R$, for all $i=1,\ldots,d-1$ and
$$ \mathbf{b}^i ( k_{i-1} , \mu_i , k_i ) =
\begin{cases}
\begin{array}{cl}
 \mathbf{u}^i_k (\mu_i ) \   &   \mbox{if }  k_{i-1} = k_i = k ,\,  i=2 , \ldots, d -1 \,  \\
0   &    \mbox{if }  k_{i-1} \not= k_i  , \, i=2 , \ldots, d -1 \, \\
 \mathbf{u}^i_k (\mu_i )  &   \mbox{if }k_i=k, \, i=1  \, \\
 \mathbf{u}^i_k (\mu_i )  &   \mbox{if }k_{i-1}=k, \, i=d  \,
\end{array}
\end{cases}
 \ . $$
{\bf Data complexity:} Let $ n :=  \max \{n_i:i=1 , \ldots , d\} $, $r := \max \{� r_{j} : j=1, \ldots, d-1  \}$. Then the number of data required for the presentation is 
$\mathcal{O} (d n r^2) $. 
Computing
a single  entry of a tensor requires the matrix  multiplication of $d$
matrices  of size at most $r\times r$. This can be performed in
$\mathcal{O} (ndr^3)$
operations.

 Since the parametrization $\tau $ can be written in the simple matrix product form (\ref{eq:mps}),
we will consider   the  TT format often  as a prototype model, and use it  frequently for  our explanations.
We remark that most properties can easily be extended to the general  hierarchical case 
with straightforward modifications \cite{hackbusch}, and we leave those modifications to the interested reader.

\subsection{Matricisation  
of a tensor and its multi-linear rank}

Let $\mathbf{u}$ be a tensor in $ \mathcal{H}_d$.
Given a fixed dimension tree $ \mathbb{T} $,  
for each  node $\alpha \in\mathbb{T}$, $\alpha \not= D$, we can build  
a  matrix $ \mathbf{U}^{\alpha}$ from $\mathbf{u}$ by grouping the  indices $\mu_i $, with $i\in \alpha$  into a row index $I$ and the remaining indices  $\mu_j$  with $ j \in D \backslash \alpha$
into the column index $J$ of the matrix 
 $ \mathbf{U}^{\alpha} = \big( \mathbf{U}^{\alpha}_{I,J}  \big) $. 
 For the root $\alpha =D$ we simply take the vectorized tensor $ \mathbf{U}^D \in \mathbb{R}^{n_1\cdots n_d \times 1}$.
 Since the rank of this matrix is one, it is often omitted. 

For example, in the Tucker case,  for  $\alpha =\{ i\} $ being a leaf,  
we set $I = \mu_i $
and $J= (\mu_1 , \ldots ,  \mu_{i-1} , \mu_{i+1},  \ldots , \mu_d)$ providing a matrix  
$$ \mathbf{U}^{\alpha}_{ \mu_i ; (\mu_1 , \ldots , \mu_{i-1}, 
\mu_{i+1},  \ldots , \mu_d) } = \mathbf{u} (   \mu_1 , \ldots, \mu_d) \ .  $$
Similar, in the TT-format, with the convention that $i\simeq \{1,\ldots,i\}$, we obtain matrices $\mathbf{U}^i \in \mathbb{R}^{n_1 \cdots n_i \times n_{i+1}\cdots n_d}$ with entries 
$$ \mathbf{U}^i_{ (\mu_1, \ldots, \mu_i); (\mu_{i+1} , \ldots, \mu_d)} = \mathbf{u} (   \mu_1 , \ldots, \mu_d)  \ . $$

\begin{definition}
Given a dimension tree $\mathbb{T}$ with $p$ nodes, we define the {\em multi-linear  rank} by the $p$-tuple 
$\mathbf{r} = ( r_{\alpha } )_{\alpha \in \mathbb{T}} $ with $r_\alpha = \operatorname{rank}( \mathbf{U}^{\alpha})$,  $ \alpha \in \mathbb{T} $.   
The set of tensors $\mathbf{u} \in \mathcal{H}_d$ of given multi-linear rank $\mathbf{r}$
will be denoted by $ \mathcal{M}_{\mathbf{r}} $. 
The set of all tensors of rank $\mathbf{s} $ at most $\mathbf{r}$,  i.e., $s_{\alpha } \leq r_{\alpha} $ for all $\alpha \in \mathbb{T}$
will be denoted by $  \mathcal{M}_{\leq \mathbf{r}} $. 
\end{definition} 

Unlike the matrix case, it is possible that for some tuples $\mathbf{r}$, $\mathcal{M}_{\mathbf{r}} = \emptyset$ \cite{Carlini}. However, since our algorithm works on a closed nonempty set $\mathcal{M}_{\leq \mathbf{r}}$, this issue does not concern us. 

In contrast to the canonical format \eqref{canonical:format}, also known as CANDECOMP/PARAFAC, see \cite{LimSilva,kolda} and the border rank problem 
\cite{landsberg}, in the present setting the rank is a well defined quantity.  
This fact makes the present concept highly attractive for tensor recovery. 
On the other hand, if a tensor $\mathbf{u}$ is of rank $\mathbf{r}$ then there exists a component tensor 
$\mathbf{b}^{\alpha} $ of the form (\ref{eq:transfer}) where  $ \ell = 1, \ldots, r_{\alpha}$. 

It is well  known  that the set of all matrices $ \mathbf{A} \in \mathbb{R}^{n\times m} $ of rank at most  $r$ is a set of common zeros of multi-variate polynomials, i.e., an algebraic variety. 
The set  $\mathcal{M}_{\leq \mathbf{r}} $ is the set of all tensors $ \mathbf{u} \in \mathcal{H}_d$,  
where the matrices $ \mathbf{U}^{\alpha} $ have a rank at most $r_{\alpha}$. Therefore, it is again a set of common zeros of multivariate polynomials.

\subsection{Higher order singular value decomposition}

Let us provide more details about the rather classical higher order singular value decomposition.
Above we have considered only binary dimension trees $\mathbb{T}$, but we can extend the considerations also to
$N$-ary trees with $N \geq 3$. The $d$-ary tree $\mathbb{T}$ (the tree with a root with $d$ sons $i \simeq \{ i\}$) induces the so-called Tucker decomposition and
the corresponding higher order singular value decomposition (HOSVD).
The Tucker decomposition was first  introduced by Tucker in $1963$ \cite{Tucker} and has been refined later on in many works, see e.g. \cite{levin, Tucker2, Tucker}. 
 \begin{definition}[Tucker decomposition]
Given a tensor $\mathbf{u} \in \mathbb{R}^{n_1  \times \cdots \times n_d}$, the decomposition

\begin{equation*}
\mathbf{u}\left( \mu_1,\ldots, \mu_d \right)=\sum_{k_1=1}^{r_1}\ldots \sum_{k_d=1}^{r_d}\mathbf{c}\left( k_1,\ldots, k_d \right) 
\mathbf{b}^1_{k_1} \left( \mu_1 \right)  \cdots \mathbf{b}^d_{k_d} \left( \mu_d \right),
\end{equation*}
$ r_i \leq n_i$, $ i=1, \ldots , d$,  is called a Tucker decomposition. The tensor $\mathbf{c} \in \mathbb{R}^{r_1 \times \cdots \times r_d}$ is called the core tensor and the
$\mathbf{b}^i_{k_i}  \in \mathbb{R}^{n_i }$, for $i=1,\ldots,d$, form a basis of the subspace $U_i \subset \mathbb{R}^{n_i}$. They can also be considered as transfer or component tensors
$ ( \mu_i, k_i ) \mapsto  \mathbf{b}^i (\mu_i , k_i )  \in  \mathbb{R}^{n_i \times r_i}$.
\end{definition}

  \begin{center}
\includegraphics{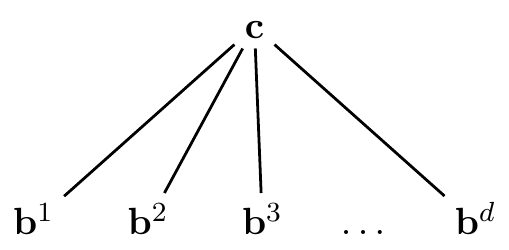} \\
{Tucker representation of an order $d$ tensor }
\label{fig:tt2}
\end{center}

   Notice that the Tucker decomposition is highly non-unique. For an $i \in \{1,\ldots,d\}$ and invertible matrix $\mathbf{Q}_i \in \mathbb{R}^{r_i \times r_i}$, one can define a matrix $\overline{\mathbf{B}}^i=\mathbf{B}^i \mathbf{Q}_i$ and the tensor $\overline{\mathbf{c}}_i$
\begin{equation*}
\overline{\mathbf{c}}_i\left( k_1,\ldots,k_d\right)=\sum_{\overline{k}_i=1}^{r_i  }   \mathbf{c}_i\left( k_1,\ldots,\overline{k}_i,\ldots k_d\right)
 \mathbf{Q}_i^{-1}\left(\overline{k}_i,k_i\right)
\end{equation*}
such that the tensor $\mathbf{u}$ can also be written as
\begin{equation*}
\mathbf{u}\left( \mu_1,\ldots, \mu_d \right)=\sum_{k_1=1}^{r_1}\ldots \sum_{k _d=1}^{r_d}\overline{\mathbf{c}}_i\left( k_1,\ldots, k_d \right) 
\mathbf{b}^1_{k_1} \left( \mu_1 \right) \cdots \overline{\mathbf{b}}^i_{k_i} \left( \mu_i \right) \cdots  \mathbf{b}^d_{k_d} \left( \mu_d \right).
\end{equation*}
 Similarly to the matrix case and the singular value decomposition, one can impose orthogonality conditions on the matrices $\mathbf{B}^i$, 
 for all $i=1,\ldots,d$, i.e., we assume that $ \{ \mathbf{b}^i_{k_i} : k_i = 1 , \ldots , r_i \}$ are orthonormal bases. 
 However, in this case one does not obtain a super-diagonal core tensor $\mathbf{c}$.

\begin{definition}[HOSVD decomposition]
The HOSVD decomposition of a given tensor $\mathbf{u} \in \mathcal{H}_d$ is a special case
of the Tucker decomposition where
\begin{itemize}
\item the bases $\{ \mathbf{b}^i_{k_i}  \in \mathbb{R}^{n_i } :  k_i =1 , \ldots , r_i\}  $ are orthogonal and normalized, for all $i=1,\ldots,d$,
\item the tensor $\mathbf{c} \in \mathcal{H}_d$ is all orthogonal, i.e., $\langle \mathbf{c}_{k_i=p},\mathbf{c}_{k_i=q}\rangle=0$, for all $i=1,\ldots,d$ and whenever $p \neq q$,
\item the subtensors of the core tensor $\mathbf{c}$ are ordered according to their  $\ell_2$ norm, i.e.,
$\left\|\mathbf{c}_{k_i=1} \right\|\geq \left\|\mathbf{c}_{k_i=2}\right\| \geq \cdots \geq \left\|\mathbf{c}_{k_i=n_i}\right\| \geq 0$.
\end{itemize}
\end{definition}
Here, the subtensor $\mathbf{c}_{k_i=p} \in \mathbb{R}^{n_1 \times \cdots \times n_{i-1} \times n_{i+1} \times \cdots \times n_d}$ is a tensor of order $d-1$ defined as
\begin{equation*}
\mathbf{c}_{k_i=p}\left( \mu_1,\ldots,\mu_{i-1},\mu_{i+1},\ldots,\mu_d\right)= \mathbf{c} \left( \mu_1, \ldots, \mu_{i-1}, p, \mu_{i+1}, \ldots ,\mu_d\right).
 \end{equation*}
 
The HOSVD can be computed via successive SVDs of appropriate unfoldings or matricisations $\mathbf{U}^{\{ i\} } = \mathbf{U}^i $, 
see e.g. \cite{kolda} and below for the more general HSVD.
For more information on this decomposition, we refer the interested reader to \cite{hosvd,hackbusch}.
 
 \subsection{Hierarchical singular value decomposition and truncation} 
 
 The singular value decomposition of the matricisation 
$ \mathbf{U}^{\alpha }$, $  \alpha \in \mathbb{T}$, is 
factorizing the tensor into two parts. Thereby,  we separate  the tree into two subtrees.  Each part
can be treated independently in  an analogous way  as before by
  applying a singular value decomposition. 
 This procedure can be continued  in a way such that one ends up
  with an explicit description of the component tensors. 
 There are several  sequential orders  
  one can proceed including top-down and bottom-up strategies. We will call  every decomposition of the 
 above type a {\em higher order singular value decomposition (HOSVD)} or in the hierarchical setting
 a {hierarchical  singular value decomposition (HSVD)}. 
  As long as  no approximation,  i.e.,
 no truncation,  has been applied during  the corresponding SVDs,  
 one obtains an exact recovery of the original tensor at the end. 
  The situation changes if we apply truncations (via thresholding). Then the result may depend on the way and the order 
   we proceed as well as on the variant of the HSVD.
 
In order to become more explicit let us demonstrate an HSVD  procedure for the model example of a TT-tensor
\cite{oseledets}, already introduced in \cite{vidal}   
for the matrix product representation. 
Without truncations this algorithm provides an exact reconstruction with a TT representation provided the multi-linear rank 
${\mathbf{s}} = (s_1,\hdots,s_{d-1})$
is chosen large enough. In general, the $s_i$'s can be chosen to be larger than the dimensions $n_i$.
Via inspecting the ranks of the relevant matricisations, the multilinear rank $\mathbf{s}$ may be determined a priori. \begin{enumerate}
\item Given: $ \mathbf{u} \in {\cal H}_d$  
of multi-linear rank ${\bf s}= (s_1, \ldots , s_{d-1} )$, $s_0= s_d:=1 .$
\item Set $\mathbf{v}^1 = \mathbf{u} $.  
\item {\bf For} $i=1, \ldots, d-1$ {\bf do }
\hspace{0.6cm} \begin{itemize}
                    \item  Form matricisation $\mathbf{V}^i \in \mathbb{R}^{s_{i-1}n_i \times n_{i+1} \cdots n_d}$ 
                    via $$\mathbf{V}^i_{ (k_{i-1} , \mu_i ) ; ( \mu_{i+1} , \ldots , \mu_d)} = \mathbf{v}^i(k_{i-1}, \mu_i,\mu_{i+1},\hdots,\mu_d).$$                    
                    \item Compute the SVD of $\mathbf{V}^i$: 
$$ \mathbf{V}^i_{ (k_{i-1} , \mu_i ) ; ( \mu_{i+1} , \ldots , \mu_d)} = \sum_{k_i=1}^{s_i} \sigma_{k_i}^{i}
                    \mathbf{b}^i  (k_{i-1} , \mu_i , k_i ) \mathbf{d}^{i+1} ( k_i , \mu_{i+1} ,     \ldots , \mu_d ), $$ \\
                    where the $ \mathbf{b}^i(\cdot,\cdot, k_i)$ and the $\mathbf{d}^{i+1}(k_i,\cdots)$ are orthonormal 
                    (i.e., the left and right singular vectors) and
                    the $\sigma_{k_i}^{i}$ are the nonzero singular values of $\mathbf{V}^i$.
                    \item Set $\mathbf{v}^{i+1}(k_i, \mu_{i+1},\cdots,\mu_d)  = \sigma_{k_i}^{i} \mathbf{d}^{i+1}(k_i, \mu_{i+1}, \ldots, \mu_d)$ 
                     \end{itemize}
  \item Set $\mathbf{b}^d( k_{d-1} , \mu_d) := \mathbf{v}^{d}  ( k_{d-1} , \mu_d) $ and $\mathbf{B}_i(\mu_i)_{k_{i-1},k_i} =
   \mathbf{b}^i( k_{i-1} , \mu_i,k_i)$ for $i=1,\hdots,d$.
  \item Decomposition $ \mathbf{u}  (\pmb{\mu})  = \mathbf{B}_1 (\mu_1) \cdots \mathbf{B}_d ( \mu_d) $.
\end{enumerate}
Above, the indices $k_i$ run from $1$ to $n_i$ and, for notational consistency, $ k_0 = k_d=1$. 
Let us notice that the present algorithm  is not the only way to use multiple singular value decompositions 
in order to obtain a hierarchical  representation of $\mathbf{u}$ for the given tree, here a TT representation. 
For example, one may  start at the right end separating $ \mathbf{b}^{d} $ first and so on. 
The procedure above  provides some {\em normal  form}  of  the tensor. 

Let us now explain hard thresholding on the example of  a TT tensor and the HSVD defined above.   
This procedure remains essentially the same with the only difference that we apply a thresholding 
to a target rank $ \mathbf{r} = (r_i)_{i=1}^{d-1}$ with $r_i \leq s_i$ at each step by setting
$\sigma^i_{k_i} = 0 $ for all $ k_i > r_i $, $i=1,\hdots,d-1$, where the $\left(\sigma^i_{k_i}\right)_{k_i} $ is the monotonically decreasing sequence of singular 
values of $ \mathbf{V}^i$. 
This leads to an approximate right factor $ \mathbf{v}^i_{\epsilon} $, within a controlled $\ell_2$ error 
$\epsilon_i = \sqrt{\sum_{k_i > r_i}  ( \sigma^i_{k_i})^2 } $.   
By the  {\em  hard thresholding}  HSVD  procedure  presented above, one obtains a unique  approximate 
tensor 
\begin{equation} \label{eq:hr}
\mathbf{u}_{\epsilon} :=\mathbf{H}_{\mathbf{r}} (\mathbf{u}) 
\end{equation}
of multi-linear rank $\mathbf{r}$ within a guaranteed error bound 
$$   
\| \mathbf{u}_{\epsilon}  - \mathbf{u}  \| \leq \sum_{i=1}^{d-1} \epsilon_i  \ . 
$$ 
In contrast to the matrix case, this approximation $\mathbf{u}_{\epsilon}$, however, 
may not be the best rank $\mathbf{r}$  approximation of $\mathbf{u} $, which is in fact NP hard to compute 
\cite{NPcomplete, Fried}. 
 A more evolved analysis shows the following quasi-optimal error bound. 

\begin{theorem} \label{thm:quasi} 
Let $ \mathbf{u}_{\epsilon} = \mathbf{H}_{\mathbf{r}} (  \mathbf{u}) $. Then there exists $  C(d) = \mathcal{O} (\sqrt{d})$,  
such that $  \mathbf{u}_{\epsilon}$ satisfies the quasi-optimal error bound
\begin{equation} \label{eq:quasi}
\inf  \{\| \mathbf{u}-   \mathbf{v}\| :  \mathbf{v}   \in \mathcal{M}_{\leq \mathbf{r}} \} 
\leq \| \mathbf{H}_{\mathbf{r}} ( \mathbf{u}) \|  
\leq C (d) \inf   \{\| \mathbf{u}-   \mathbf{v}\| : \mathbf{v}  \in \mathcal{M}_{\leq\mathbf{r}} \}  \ . 
\end{equation}
The constant satisfies $C(d)=\sqrt{d}$ for the Tucker format \cite{lars}$, C(d) = \sqrt{d-1} $ for the TT format \cite{oseledets1} and $C(d) = \sqrt{2d-3 } $ for a balanced tree in the HSVD of \cite{lars}.
\end{theorem}

The procedure introduced above can be modified to apply for general hierarchical tensor representations.
When we consider the  HSVD in  the sequel, 
we have in mind that we have fixed our hierarchical  SVD method  choosing one of the several  variants.

\subsection{Hierarchical tensors as differentiable manifolds} 

It has been shown that, fixing a tree $\mathbb{T}$, the set of hierarchical tensors of exactly multi-linear rank 
$\mathbf{r}$ forms an analytical manifold \cite{hrs-tt,uv,falco-hackbusch-nouy,lrsv}.  
We will describe its essential features using 
the TT format or the matrix product representation. 
For an invertible $r_1 \times r_1 $ matrix $ \mathbf{G}_1$, it holds 
\begin{eqnarray*}
\mathbf{u}  (\pmb{\mu})  &  = & {\mathbf{B}^{1} (\mu_{1})   
\mathbf{B}^{2} (\mu_{2})  \cdots \mathbf{B}^{i} (\mu_{i }) \cdots \mathbf{B}^{d} (\mu_{d})} \\
  &  = & {\mathbf{B}^{1} (\mu_{1})  \mathbf{G}_1 \mathbf{G}_1^{-1} 
\mathbf{B}^{2} (\mu_{2})  \cdots \mathbf{B}^{i} (\mu_{i }) \cdots \mathbf{B}^{d} (\mu_{d})} \\
&  = & { \widetilde{\mathbf{B}}^{1}  (\mu_{1})  \widetilde{\mathbf{B}}^{2} (\mu_{2})  \cdots \mathbf{B}^{i} (\mu_{i }) \cdots \mathbf{B}^{d} (\mu_{d})},
\end{eqnarray*}
where $\widetilde{\mathbf{B}}^1(\mu_1)=\mathbf{B}^1(\mu_1)\mathbf{G}_1$ and $\widetilde{\mathbf{B}}^2(\mu_2)=\mathbf{G}_1^{-1}\mathbf{B}^2(\mu_2)$.
This provides two different representations of the same tensor $\mathbf{u}$.
In order to remove the redundancy in the above parametrization of the set $ \mathcal{M}_{\bf r }$, let us consider the linear space of parameters
$ (\mathbf{b}^1 , \ldots , \mathbf{b}^d)  \in \mathcal{X} :=  \times_{i=1}^d  {X}_i $, $ {X_i} := \mathbb{R}^{r_{i-1} r_i n_i } $, or  equivalently 
$ \mathcal{U} :=  \big( \mathbf{B}^1 (.), \ldots , \mathbf{B}^d (.) \big)$,   together with  a Lie group action. 
 For  a collection of  invertible matrices $\mathcal{G} =  (\mathbf{G}_1, \ldots , 
\mathbf{G}_{d-1} )$   we define a transitive   group action by 
$$ \mathcal{G} \circ \mathcal{U} := \big( \mathbf{B}^1  \mathbf{G}_1, \mathbf{G}_1^{-1} \mathbf{B}^2  \mathbf{G}_2 , \
\ldots , \mathbf{G}_{d-1}^{ -1} \mathbf{B}^d \big) \ .  $$
One  observes  that the tensor $\mathbf{u}$ remains unchanged  under this transformation of the component tensors.
 Therefore, we will identify two representations $ \mathcal{U}_1  \sim \mathcal{U}_ 2$,  
  if there exists $ \mathcal{G} $ such that $ \mathcal{U}_2 = \mathcal{G} \circ \mathcal{U}_1 $.
  It is easy to see that   the  equivalence classes 
$  [ \mathcal{U} ] : = \{\mathcal{V} :  \mathcal{U} \sim \mathcal{V} \}$ define  smooth manifolds in $\mathcal{X}$.  
We are interested in the quotient manifold $ \mathcal{X}/ \sim $, which is isomorphic to $\mathcal{M}_{\mathbf{r}}$.    
This construction gives rise to an embedded analytic manifold  \cite{absil,lrsv,uv} where 
 a Riemannian metric is canonically defined.

The {\em  tangent space} $ \mathcal{T}_{\mathbf{u}} $ at $\mathbf{u}  \in \mathcal{M}_{\bf r} $
is of importance for calculations. It can be easily determined by means of  the product rule as follows. 
A generic tensor $ \delta \mathbf{u}  \in \mathcal{T}_{\mathbf{u}}$ of a TT tensor $\mathbf{u}$
 is of the form  
\begin{eqnarray*}
 \delta \mathbf{u} ( \mu_1 , \ldots , \mu_d ) & = & \mathbf{t}^1  ( \mu_1 , \ldots , \mu_d )  
 + \ldots + \mathbf{t}^d (\mu_1, \ldots , \mu_d ) \\
& = &  \delta{\bf B}^1 (\mu_1) {\bf B}^2 (\mu_2) \cdots {\bf B}^d (\mu_d ) + \ldots  \\
 & + &   {\bf B}^1 (\mu_1) \cdots {\bf B}^{i-1} (\mu_{i-1} )  \delta {\bf B}^i  (\mu_i ) {\bf B}^{i+1} (\mu_{i+1} ) \cdots  
 {\bf B}^d (\mu_d)  + \ldots \\
 & + &  
{  \bf B}^1 (\mu_1) \cdots {\bf B}^{d-1} (\mu_{d-1} )  \delta {\bf B}^d  (\mu_d) \ . 
\end{eqnarray*}
 This tensor is uniquely determined if we impose {\em gauging conditions}
 onto $ \delta \mathbf{B}^i $, $i=1, \ldots , d-1$ \cite{hrs-tt,uv}.  There is no gauging condition imposed onto $ \delta \mathbf{B}^d$. 
 Typically these conditions are of the form
\begin{eqnarray}
\sum_{k_{i-1} =1}^{r_{i-1}} \sum_{\mu_i =1}^{n_i} {\mathbf{b}^i (k_{i-1}, \mu_i , k_i ) } \,  \delta \mathbf{b}^i (k_{i-1},\mu_i , k'_i ) 
 & = & 0 \ ,  \  \forall \, k_i , k_i' = 1, \ldots , r_i \ . \label{eq:gauge}
\end{eqnarray}
With this condition at hand an orthogonal projection $P_{\mathcal{T}_{\mathbf{u}}} $  onto the tangent space 
$\mathcal{T}_{\mathbf{u}} $ 
is well defined  and computable in a straightforward way.

The manifold $\mathcal{M}_{\bf r} $ is open and its closure is $\mathcal{M}_{\leq \mathbf{r}} $, the set of  all tensor{s} 
with ranks  at most ${r}_{\alpha} $, $ \alpha  \in \mathbb{T}$,
$$  \mbox{clos }  \big(  \mathcal{M}_{\bf r} \big) = \mathcal{M}_{\leq \mathbf{r}} \ . $$ 
This important result  is based on the observation that the matrix rank is an upper semi-continuous function \cite{hackbusch-falco}.
The  singular points are exactly  those for which at least one rank  $\tilde{r}_{\alpha}  < r_{\alpha}$ is not maximal, see e.g \cite{su2}.
We remark that,  for the root $D$ of the partition tree $\mathbb{T}$,  there is no gauging condition imposed onto $ \delta \mathbf{B}^D$.
 We highlight  the following facts without explicit proofs for hierarchical tensors.
 
\begin{proposition}\label{prop:Riemannian} 
Let $\mathbf{u} \in \mathcal{M}_{ \mathbf{r}}$. Then 
 \begin{itemize}
\item[(a)] the corresponding  gauging conditions (\ref{eq:gauge}) imply  that the  tangential vectors 
$ \mathbf{t}^i $ are pairwise orthogonal;
\item[(b)] the tensor $\mathbf{u}$ is  included in its own tangent space $\mathcal{T}_{\mathbf{u}} $;
\item[(c)] the multi-linear rank of a tangent vector is at most 
$2 \mathbf{r}$, i.e.,   $ \delta \mathbf{u} \in \mathcal{M}_{\leq 2 \mathbf{r}} $. 
\end{itemize}
\end{proposition}
Curvature estimates are given in \cite{lrsv}.

 \section{Tensor completion for hierarchical tensors} 

\subsection{The low rank tensor recovery problem}

We pursue on 
extending  methods for solving optimization problems in the calculus of hierarchical tensors 
to {\em low rank tensor recovery}  
and to {\em tensor  completion} as a special case.
The latter   builds on ideas from the theory of
 {\em compressed sensing} \cite{rauhut},  which predicts that sparse vectors can be recovered efficiently from incomplete linear 
 measurements via efficient algorithms. 
Given a linear measurement mapping $\mathcal{A}  :  \mathcal{H}_d = \bigotimes_{i=1}^d  \mathbb{R}^{n_i} \to 
 \mathbb{R}^m$ our aim is to recover a tensor $ \mathbf{u}  \in \mathcal{H}_d$ 
 from $m \ll N:= n_1\cdot n_2 \cdots n_d$ measurements $ \mathbf{b} \in \mathbb{R}^m$ given by
  \begin{equation*}
\mathbf{ b} = (b_i)_{i=1}^m  = \mathcal{A} \mathbf{ u}  \ . 
\end{equation*} 
Since this problem is underdetermined we additionally assume that $\mathbf{u}$ is of low rank, i.e.,
given a dimension tree $\mathbb{T}$  and a multi-linear rank $ \mathbf{r}$, we suppose that the tensor is 
contained in the corresponding tensor manifold, $\mathbf{u}  \in \mathcal{M}_{\mathbf{r}}$. 

The {\em tensor completion} problem -- generalizing the matrix completion problem \cite{care09,cata10,gr11} --
is the special case where the measurement map subsamples entries of the tensor,
i.e., 
$ b_i= \big( \mathcal{A} \mathbf{ u} \big)_i =  \mathbf{u} (\pmb{\mu}_i)=\mathbf{u} ( \mu_{1,i} , \ldots , \mu_{d,i} ) $, $ i =1 , \ldots , m$, with the
multi-indices $\pmb{\mu}_i$ being contained in a suitable index set $\Omega \subset [n_1] \times \cdots \times [n_d]$ of cardinality $m \ll n_1 \cdots n_d$.

We remark that in practice the desired rank $\mathbf{r}$ may not be known in advance and/or the tensor $\mathbf{u}$ is only close 
to $\mathcal{M}_{\mathbf{r}}$ rather than being exactly contained in $\mathcal{M}_{\mathbf{r}}$. Moreover,
the left hand side $\mathbf{b}$ may not be known exactly because of noise on the measurements.
In the present paper we defer from tackling these important stability and robustness issues and focus on the problem in the above form. 
 
The problem of reconstructing ${\mathbf{u}} \in \mathcal{H}_d$ from $\mathbf{b} = {\cal A} {\mathbf{u}}$ can be reformulated
to finding the minimizer of
 \begin{equation} \label{eq:opt}
 \mathcal{J} (\mathbf{v} ) =  \frac{1}{2} \| \mathcal{A} \mathbf{ v} - \mathbf{b} \|^2
\quad \mbox{ subject to } \mathbf{v} \in \mathcal{M}_{\mathbf{r}} \ .
\end{equation}
In words, we are looking for a tensor of multi-linear rank $\mathbf{r} $, which fits best the given measurements. A minimizer over $\mathcal{M}_{\leq \mathbf{r}}$ always exists, but a solution of the above problem 
 may not exist  in general since $\mathcal{M}_{\mathbf{r}} $ is not closed. 
 However, assuming ${\mathbf{u}} \in \mathcal{M}_{\mathbf{r}}$ and $\mathbf{b} = {\cal A}\mathbf{u}$ as above,
 existence of a minimizer is trivial because setting
 $\mathbf{v} = \mathbf{u}$ gives   $\mathcal{J} (\mathbf{v} ) = 0$. 
We note that finding a minimizer of \eqref{eq:opt} is NP-hard in general \cite{NPcomplete,Fried}. 
  
The necessary  first order condition for a minimizer of the problem
(\ref{eq:opt}) can be formulated as  follows, see e.g.~\cite{lrsv}.  
If $ \mathbf{u}   \in \mathcal{M}_{\bf r}  $ is a solution of $\mbox{argmin}_{\mathbf{v}\in \mathcal{M}_{\mathbf{r}}} \mathcal{J}\left(\mathbf{v}\right)$,  then
\begin{equation*}  
\langle \nabla  \mathcal{J} ( \mathbf{u} )  , \delta\mathbf{u} \rangle  = 0 \ , \ \ \mbox{ for all } \delta\mathbf{u} 
 \in \mathcal{T}_{\mathbf{u}},
\end{equation*}
where $\nabla \mathcal{J} $ is the gradient of $\mathcal{J}$. 

\subsection{Optimization approaches}

In analogy to compressed sensing and low rank matrix recovery where convex relaxations ($\ell_1$-minimization and nuclear norm minimization)
are very successful, a first idea for a tractable alternative to \eqref{eq:opt} may be to find an analogue of the nuclear norm for the tensor case.
A natural approach is to consider the set $Q$ of unit norm rank one tensors in ${\cal H}_d$,
\[
Q = \{ \mathbf{u} = \mathbf{b}_1 \otimes \mathbf{b}_2 \cdots \otimes \mathbf{b}_d : \| \mathbf{u} \| = 1 \} \ . 
\]
Its closed convex hull $B = \overline{\operatorname{conv} Q}$ is taken as the unit ball of the tensor nuclear norm, so that
the tensor nuclear norm is the gauge function of $B$,
\[
\| \mathbf{u} \|_* = \inf \{ t: \mathbf{u} \in tB\} \ .
\]
In fact, for the matrix case $d=2$ we obtain the standard nuclear norm. Unfortunately, for $d \geq 3$, the nuclear tensor
norm is NP hard to compute \cite{NPcomplete,Fried}.

The contributions \cite{recht,goldfarb,liu1} proceed differently by considering the matrix nuclear norm of several unfoldings of the tensor ${\mathbf{u}}$. 
Given a dimension tree $\alpha \in \mathbb{T}$ and corresponding matricisations $\mathbf{U}^{\alpha} $ of ${\mathbf{u}}$,
we consider the Schatten norm
$$  \| \mathbf{U}^{\alpha} \|_{p} := 
 \bigg( \sum_{k_{\alpha}} ( \sigma^{\alpha}_{ k_{\alpha } })^p  \bigg)^{\frac{1}{p} } , \ \alpha \in \mathbb{T}, \quad 1 \leq p < \infty \ ,$$
where the  $\sigma^{\alpha}_{ k_{\alpha}}$ are the singular values of $\mathbf{U}^{\alpha}$, see e.g.\  \cite{bh97,su1}.
Furthermore, for $1 \leq q \leq \infty$ and given  $ a_\alpha> 0 $, e.g. $a_\alpha =1$, a   
norm on $ \mathcal{H}_d$  can be introduced by 
 $$ \| \mathbf{u} \|_{p,q}^q : =  \sum_{\alpha \in \mathbb{T} }    a_\alpha \| \mathbf{U}^{\alpha} \|_{p}^q \, .$$
 A prototypical choice of a convex optimization formulation  used  for tensor recovery consists in finding
 $$
 \mbox{argmin} \{   \mathcal{J} ( \mathbf{u}) := \| \mathbf{u} \|_{1,q} : \mathcal{A} \mathbf{u} = \mathbf{b} \} \ . 
 $$
 For $a_\alpha =1$ and $q=1$ this functional was suggested in \cite{recht,liu1} for reconstructing tensors in the Tucker format.
 Although the numerical results are reasonable, it seems that conceptually this is not the ``right'' approach and too simple for 
 the present purpose, see corresponding negative results in \cite{goldfarb2}.
 For the Tucker case first rigorous results have been shown in \cite{goldfarb}. However the approach followed there is based on results
 from matrix completion and does not use the full potential of tensor decompositions. 
 In fact, their bound on the number of required measurements is $ \mathcal{O} ( r n^{d-1} ) $. 
A more ``balanced'' version of this approach is considered in \cite{goldfarb2}, where the number of required measurements
scales like $ \mathcal{O} ( r^{d/2} n^{d/2} ) $, which is better but still far from the expected linear scaling in $n$, see also 
Theorem~\ref{thm:TRIP:Gaussian} below.

\subsection{Iterative hard thresholding  schemes} 

Rather than following the convex optimization approach which leads to certain difficulties as outlined above, 
we consider versions of the iterative hard thresholding algorithm
well-known from compressed sensing \cite{blda08,rauhut} and low rank matrix recovery \cite{tanner}.
Iterative hard thresholding algorithms fall into the larger class of {\em projected gradient methods}. Typically 
one  performs a gradient step in the ambient space  $\mathcal{H}_d$ , followed
by a mapping  $ \mathcal{R} $
onto the set of low rank tensors $ \mathcal{M}_{ \mathbf{r}}$ or $\mathcal{M}_{\leq \mathbf{r}}$,  formally
  \begin{eqnarray*}
 \mathbf{ y}^{n+1} & : = & \mathbf{u}^{n} - \alpha_n 
 \nabla  \mathcal{J} ( \mathbf{ u}^n )    \quad \mbox{(gradient step)} \nonumber  \\
 &  = & \mathbf{u}^{n} - \alpha_n \big( \mathcal{A}^* 
 ( \mathcal{A} \mathbf{ u}^n - \mathbf{b}  )   \big) \ ,    \label{eq:grad}\\
  \mathbf{u}^{n+1} & : = & \mathcal{R} (\mathbf{y}^{n+1})   \quad  \mbox{(projection step)}\ .    \label{eq:nonpro}
\end{eqnarray*}
Apart from specifying the steplength $\alpha_n$, the above algorithm depends on the choice of the projection  operator 
$\mathcal{R} : \mathcal{H}_d \to \mathcal{M}_{\leq \mathbf{r}}$. 
An example would be
$$ \mathcal{R}( \mathbf{y}^{n+1} ):   = \mbox{argmin} \{ \| \mathbf{y}^{n+1}  - \mathbf{z}�\| : \mathbf{z} \in \mathcal{M}_{\leq \mathbf{r}} \} 
\ .$$
Since this projection  is not computable in general \cite{NPcomplete,Fried},  we may rather choose the hierarchical  singular value (HSVD) 
thresholding procedure  (\ref{eq:hr})   
\begin{equation*} 
  \mathbf{u}^{n+1} := \mathcal{R} (\mathbf{y}^{n+1} )  = \mathbf{H}_{\mathbf{r}} (\mathbf{y}^{n+1} ) \quad    \mbox{ (hard thresholding)} \ , 
\end{equation*}
which is only quasi-optimal (\ref{eq:quasi}). We will call this procedure {\em tensor iterative hard thresholding  (TIHT)}, or shortly {\em iterative hard thresholding (IHT)}.  

Another possibility for the projection operator relies on the concept of retraction from differential geometry \cite{absil}.
A retraction maps $ \mathbf{u} + \boldsymbol{\xi} $, where 
 $ \mathbf{u}  \in  \mathcal{M}_{ \mathbf{r}} $ and $ \boldsymbol{\xi} \in 
 \mathcal{T}_{ \mathbf{u}} $,  smoothly  to the manifold.   
 For $R :   (  \mathbf{u} , \boldsymbol{\xi} )  \to  R (  \mathbf{u} , \boldsymbol{\xi} ) \in  \mathcal{M}_{ \mathbf{r}}   $ 
 being a retraction it is required that $ R$ 
 is twice differentiable  and $ R ( \cdot, \mathbf{0} ) = \mathbf{I} $ is the identity.   
 Moreover, a retraction satisfies, for $\|\xi\|$ sufficiently small, 
\begin{eqnarray}
  \|    \mathbf{u} + \boldsymbol{\xi} - R (  \mathbf{u} , \boldsymbol{\xi} ) \|�&=& \mathcal{O} ( \| \boldsymbol{\xi} \|^2) \ , \label{retract:prop1}\\
   \|    \mathbf{u}  - R (  \mathbf{u} , \boldsymbol{\xi} ) \|�& = &  \mathcal{O} ( \| \boldsymbol{\xi} \|) \ . \notag 
\end{eqnarray}
Several examples of retractions for hierarchical  tensors are known \cite{lrsv,kressner}, which can be efficiently computed.
If a retraction is available, then a  nonlinear projection $\mathcal{R} $ can be realized in two steps.  
First we project (linearly)  onto the tangent space $\mathcal{T}_{\mathbf{u}^n}$ at $\mathbf{u}^n $,  and afterwards we apply a
{\em retraction}  $R$. This leads to the so-called  {\em Riemaniann gradient iteration method}  (RGI) defined formally as
 \begin{eqnarray*}
 \mathbf{ z}^{n+1} & : = &P_{\mathcal{T}_{\mathbf{u}^n}} \big(
  \mathbf{u}^{n} - \alpha_n P_{\mathcal{T}_{\mathbf{u}^n}} \big( \mathcal{A}^* 
 ( \mathcal{A} \mathbf{ u}^n - \mathbf{b} )   \big)\big) \ \ \mbox{(projected gradient step)}   \label{eq:linpro} \\
&  = &P_{\mathcal{T}_{\mathbf{u}^n}} \big(
  \mathbf{u}^{n} - \alpha_n \mathcal{A}^* 
 ( \mathcal{A} \mathbf{ u}^n - \mathbf{b} )   \big)  
\nonumber 
   = :  \mathbf{u}^{n} +  \boldsymbol{\xi}^n  \label{eq:rgi}  \\
  \mathbf{u}_{n+1} & : = &  {R} (\mathbf{u}^n ,   \mathbf{z}^{n+1} - \mathbf{u}^n )=  {R} (\mathbf{u}^n ,   \boldsymbol{\xi}^n )  
  \label{eq:retract} \ \ 
    \mbox{(retraction step)} .  
\end{eqnarray*}
With a slight abuse of notation we will write  $$\mathcal{R} (\mathbf{y}^{n+1})   = R \circ P_{\mathcal{T}_{\mathbf{u}^n}} \mathbf{y}^{n+1}$$ for the RGI.
 
It may happen that an iterate $ \mathbf{u}^n $ is  of lower rank, i.e., $ \mathbf{u}^n \in \mathcal{M}_{\mathbf{s}}$ with
$s_\alpha < r_\alpha$ at least for one  $\alpha \in \mathbb{T}$. 
In this case $\mathbf{u}^n \in \mathcal{M}_{\leq \mathbf{r}}$ is a singular point and  no longer on our manifold, i.e., 
$ \mathbf{u}^n  \not\in \mathcal{M}_{\mathbf{r}}$, 
  and our 
  RGI algorithm  
  fails. However, since $\mathcal{M}_{\mathbf{r}}$ is dense in $\mathcal{M}_{\leq \mathbf{r}} $, for arbitrary 
  $\epsilon >0$, there exists $ \mathbf{u}_{\epsilon}^n \in \mathcal{M}_{\mathbf{r}}$, with
  $\|  \mathbf{u}^n - \mathbf{u}_{\epsilon}^n \|< \epsilon $. Practically such a regularized $\mathbf{u}^n_{\epsilon} $ is not hard to choose. 
  Alternatively,  the algorithm described above 
  may be regularized in a sense that it automatically avoids the situation being trapped in a singular point \cite{kressner}.
Here, we do not  go into these  technical details.

\subsection{Restricted isometry property for hierarchical tensors}

A crucial sufficient condition for exact recovery in compressed sensing is the {\em restricted isometry property (RIP)}, see Chapter 1.
It has been applied in the analysis of iterative hard thresholding both in the compressed sensing setting \cite{blda08,rauhut}
as well as in the low rank matrix recovery setting \cite{tanner}.
The RIP can be easily 
generalized to the present tensor setting. 
As common, $ \| \cdot \| $ denotes the Euclidean norm below. 

\begin{definition}
Let $\mathcal{A}: \mathcal{H}_d=\bigotimes_{i=1}^d\mathbb{R}^{n_i} \rightarrow \mathbb{R}^m $ be a linear measurement map, $\mathbb{T}$
be a dimension tree and for $\mathbf{r} = (r_\alpha)_{\alpha \in \mathbb{T}}$, let ${\mathcal M}_{\mathbf{r}}$ be the associated low rank tensor manifold.
The tensor restricted isometry constant (TRIC)  
$\delta_{\mathbf{r}}$ of $\mathcal{A}$
 is the smallest number such that
 \begin{equation} \label{eq:TRIP}
(1 - \delta_{\mathbf{r}} ) \| \mathbf{u}  \|^2 \leq \| \mathcal{A} \mathbf{u} \|^2 \leq  (1+ \delta_{\mathbf{r}} ) \| \mathbf{u} \|^2,  \quad \mbox{
 for all } \mathbf{u} \in {\mathcal M}_{\mathbf{r}}.
\end{equation}
\end{definition}
Informally, we say that a measurement map ${\cal A}$ satisfies the {\it tensor restricted isometry property} (TRIP)
if $\delta_{\mathbf{r}}$ is small (at least $\delta_{\mathbf{r}} < 1$) for some ``reasonably large'' $\mathbf{r}$.
 
Observing that $ \mathbf{u}  + \mathbf{v} \in \mathcal{M}_{\leq2\mathbf{r}} $
for two tensors $\mathbf{u}, \mathbf{v}   \in   \mathcal{M}_{\leq\mathbf{r}} $ 
the TRIP (\ref{eq:TRIP}) of order  $2 \mathbf{r}$ 
implies that 
$ \mathcal{J} $ has a unique minimizer on $\mathcal{M}_{\mathbf{r}}$.   
Indeed,  for two tensors $ \mathbf{u}_1   , \mathbf{u}_2 \in   \mathcal{M}_{\mathbf{r}}$ satisfying
$  \mathcal{A} \mathbf{ u}_1 - \mathbf{b} = \mathcal{A} \mathbf{ u}_2 - \mathbf{b}  = \mathbf{0}$, it follows that
$\mathcal{A} ( \mathbf{u}_1 - \mathbf{u}_2  ) = \mathbf{0} $ in contradiction to the TRIP and 
$ \mathbf{ u}_1 - \mathbf{u}_2 \in  \mathcal{M}_{\leq2\mathbf{r}} $.

For Gaussian (or more generally subgaussian) measurement maps, 
the TRIP holds with high probability for both the HOSVD and the TT format  under a suitable bound on the number of measurements \cite{rss1,rss2}, which basically scales like the number of degrees of freedom of a tensor of multi-linear rank $\mathbf{r}$ (up to a logarithmic factor in $d$).
In order to state these results, let us introduce Gaussian measurement maps. A measurement map ${\cal A}: \mathcal{H}_d \rightarrow \mathbb{R}^m$ can be identified with a tensor in $\mathbb{R}^{m} \otimes \bigotimes_{i=1}^d \mathbb{R}^{n_i}$
via
\[
({\cal A}{\mathbf{u}})_\ell = \sum_{\mu_1=1}^{n_1} \sum_{\mu_2=1}^{n_2} \cdots \sum_{\mu_d=1}^{n_d} {\mathbf{a}}(\ell,\mu_1,\mu_2,\hdots,\mu_d) \mathbf{u}(\mu_1,\hdots,\mu_d), \quad \ell=1,\hdots,m \ .
\] 
If all entries of ${\cal A}$ are independent realizations of normal distributed random variables with mean zero and variance $1/m$, 
then ${\mathcal A}$ is called a {\it Gaussian measurement map}.  
\begin{theorem}[\cite{rss1,rss2}]\label{thm:TRIP:Gaussian} 
For $\delta, \varepsilon \in \left(0,1\right)$, a random draw of a Gaussian measurement map $\mathcal{A} : \mathcal{H}_d=\bigotimes_{i=1}^d \mathbb{R}^{n_i} \rightarrow \mathbb{R}^m$ satisfies $\delta_{\mathbf{r}}\leq \delta$ with probability at least $1-\varepsilon$ provided
\begin{itemize}
\item HOSVD format: $m \geq C \delta^{-2}\max \left\{\left({r}^d+d{n}{r}\right) \log (d), \log\left(\varepsilon^{-1}\right)\right\},$ 
\item TT format: $m \geq C \delta^{-2}\max \left\{\left(d{n}{r}^2\right) \log \left(d {r}\right), \log\left(\varepsilon^{-1}\right)\right\},$ 
\end{itemize}
where  ${n}=\max\left\{n_i: i =1,\ldots,d\right\}$ and ${r}=\max\left\{r_i: i =1,\ldots,d\right\}$ 
and $C>0$ is a  universal constant.
\end{theorem}
The above result extends to subgaussian and in particular to Bernoulli measurement maps, see e.g.\ \cite{rauhut,Vershynin} for the definition
of subgaussian random variables and matrices.
Presently, it is not clear whether the logarithmic factor in $d$ above is necessary or whether it is an artefact of the proof.
We conjecture that similar bounds hold also for the general hierarchical tensor format.
We note that 
in practice, the application of Gaussian sensing operators acting on the tensor product space $\mathcal{H}_d$ seems to be  
computationally too expensive
except for  relatively small  dimensions $d$,  (e.g. $d=2,3,4$), and small $n_i$.     
A more realistic measurement map for which TRIP bounds can be shown \cite{rss2} is the decomposition
of random sign flips of the tensor entries, a $d$-dimensional Fourier transform and random subsampling.
All these operations can be performed quickly (exploiting) the FFT.  For further  details we refer to \cite{rss1,rss2}. 
 
We finally remark that the TRIP does not hold in the tensor completion setup because sparse and low rank tensor may belong to the kernel
of the measurement map. In this scenario, additional incoherence properties on the tensor to be recovered like in the matrix completion scenario \cite{care09,gr11,re12} are probably necessary.

\subsection{Convergence results} 

Unfortunately, a full convergence (and recovery) analysis of the TIHT and RGI algorithms under the TRIP is not yet available.
Nevertheless, we present two partial results. The first concerns the local convergence of the RGI and the second 
is a convergence analysis of the TIHT under an additional assumption on the iterates.

We assume that $\mathbf{u} \in {\mathcal M}_{\mathbf{r}}$, where the low rank tensor manifold is associated to a fixed 
hierarchical tensor format. Measurements are given by
\[
{\mathbf{b}} = {\mathcal A} \mathbf{u}\ ,
\]
where $\mathcal{A}$ is assumed to satisfy the TRIP of order $3 \mathbf{r}$ below.
Recall that our projected gradient scheme starts with an initial guess $\mathbf{u}^0$ and forms the iterates
\begin{align} \label{eq:yn}
 \mathbf{y}^{n+1} & :=  \mathbf{u}^n   +  \mathcal{A}^* ( \mathbf{b} - \mathcal{A} \mathbf{u}^n ) \ , \\
\mathbf{u}^{n+1} &: = \mathcal{R} ( \mathbf{y}^{n+1})  \ , \label{eq:un}
\end{align}
where either $ \mathcal{R} (\mathbf{u}^{n+1} ) : = \mathbf{H}_{\mathbf{r}} \left(\mathbf{y}^{n+1}\right)$ (TIHT) or 
$  \mathcal{R} (\mathbf{u}^{n+1} )  : = R \circ P_{\mathcal{T}_{\mathbf{u}^n}} \mathbf{y}^{n+1}  \ $ (RGI).

We first show local convergence, in the sense that the iterates ${\mathbf{u}}^n$ converge
to the original tensor $\mathbf{u}$ if the initial guess is sufficiently close to 
the solution $\mathbf{u} \in \mathcal{M}_{\mathbf{r}}$. Of course, this analysis also applies if one of the later iterates
comes close enough to $\mathbf{u}$.

\begin{theorem}[Local convergence] \label{th:local}
Let $\mathbf{b} =\mathcal{A}\mathbf{u}$ for $\mathbf{u} \in \mathcal{M}_{\leq\mathbf{r}}$ and let ${\mathbf{u}^n}$
be the iterates  \eqref{eq:yn}, \eqref{eq:un} of the Riemannian gradient iterations, i.e., 
$\mathcal{R} (\mathbf{u}^{n+1} )  : = R \circ P_{\mathcal{T}_{\mathbf{u}^n}} \mathbf{y}^{n+1}$, where
$R$ is a retraction.
In addition, let's assume that $ \mathcal{A} $ satisfies the TRIP of order ${3 \mathbf{r}}$,    
i.e., $\delta_{3\mathbf{r}} \leq \delta < 1$. 
 Suppose that  
 $$ \| \mathbf{u} - \mathbf{u}^0 \| \leq  \varepsilon$$ is sufficiently small and the distance
 to the singular points $ \varepsilon <  \mbox{dist} ( \mathbf{u}^0, \partial \mathcal{M}_{\bf r} )$ is sufficiently large.
 Then, there exists $ 0 < \rho <1 $ (depending on $\delta$ and $\varepsilon$) such that    
the series 
$ \mathbf{u}^n \in \mathcal{M}_{\leq \mathbf{r}}$ convergences linearly
 to $\mathbf{u}\in \mathcal{M}_{\leq \mathbf{r}}$ with rate $\rho$, 
$$
\| \mathbf{u}^{n+1}   - \mathbf{u} \| \leq  \rho \| \mathbf{u}^n - \mathbf{u}\|. $$
\end{theorem}

\begin{proof}
 We consider
 the orthogonal projection $ P_{\mathcal{T}_{\mathbf{u}^n}   }$ onto the tangent space $\mathcal{T}_{\mathbf{u}^n}  $.
 There exists   $1<  \gamma = \gamma  (\varepsilon) $ and $   \kappa >0 $ depending on the curvature of 
 $\mathcal{M}_{\mathbf{r}}$,
 such that, for all $ \|  \mathbf{v} - \mathbf{u}^n \| < \varepsilon $, it holds that \cite{lrsv}
 \begin{align}
\gamma^{-1} \   \|\mathbf{u}^n -  \mathbf{v}  \|  & \leq   \| P_{\mathcal{T}_{\mathbf{u}^n} }
(\mathbf{u}^n -    \mathbf{v} ) \| 
 \leq \gamma \  \|\mathbf{u}^n -  \mathbf{v}  \|  \label{eq:equiv1} \\ 
 \| ( I -  P_{\mathcal{T}_{\mathbf{u}^n} })  ( \mathbf{u}^n  -    \mathbf{v})   \|  & \leq  
 \kappa \   \| P_{\mathcal{T}_{\mathbf{u}^n} } ( \mathbf{u} -    \mathbf{v})   \|^2   \ .  \label{eq:square}
\end{align}
Using the triangle inequality we estimate
\begin{align}
& \left\| \mathbf{u}^{n+1} - \mathbf{u}\right\|  =  
\left\| R\left(\mathbf{u}^n, P_{\mathcal{T}_{\mathbf{u}^n}}\mathbf{y}^{n+1} -\mathbf{u}^n \right) - \mathbf{u}\right\| \nonumber \\
&\leq \left\| R\left(\mathbf{u}^n, P_{\mathcal{T}_{\mathbf{u}^n}}\mathbf{y}^{n+1} -\mathbf{u}^n \right) - P_{\mathcal{T}_{\mathbf{u}^n}}\mathbf{y}^{n+1} \right\|  
+ \left\| P_{\mathcal{T}_{\mathbf{u}^n}}\mathbf{y}^{n+1} - P_{\mathcal{T}_{\mathbf{u}^n}}\mathbf{u} \right\|  + \left\| P_{\mathcal{T}_{\mathbf{u}^n}}\mathbf{u}-\mathbf{u} \right\|.
\label{3terms}
\end{align}
We will bound each of the three terms in \eqref{3terms} separately.
We start with the first term, where we exploit the property \eqref{retract:prop1} of retractions.
Moreover, 
$W^n  :=\mathcal{T}_{\mathbf{u}^n}
\subset  \mathcal{M}_{\leq  2 \mathbf{r}} $ by Proposition~\ref{prop:Riemannian}(c)
and  $P_{\mathcal{T}_{\mathbf{u}^n}}\mathbf{y}^{n+1} = P_{\mathcal{T}_{\mathbf{u}^n}} \mathbf{u}^n + P_{\mathcal{T}_{\mathbf{u}^n}} \mathbf{u}^n
{\mathcal A}^* {\mathcal A}(\mathbf{u} - \mathbf{u}^n) = \mathbf{u}^n + P_{\mathcal{T}_{\mathbf{u}^n}} 
{\mathcal A}^* {\mathcal A}(\mathbf{u} - \mathbf{u}^n)$ because $P_{\mathcal{T}_{\mathbf{u}^n}}  \mathbf{u}^n  = \mathbf{u}^n$
by Proposition~\ref{prop:Riemannian}(a).
Since $\mathbf{u} - \mathbf{u}^n \in {\cal M}_{\leq 2 \mathbf{r}}$ we may apply the TRIP of order $2{\mathbf{r}} < 3 \mathbf{r}$ to obtain
\begin{align}
&\left\| R\left(\mathbf{u}^n, P_{\mathcal{T}_{\mathbf{u}^n}}\mathbf{y}^{n+1} -\mathbf{u}^n \right) - P_{\mathcal{T}_{\mathbf{u}^n}}\mathbf{y}^{n+1} \right\| \leq  C \left\| P_{\mathcal{T}_{\mathbf{u}^n}}\mathbf{y}^{n+1}-\mathbf{u}^n\right\|^2  \nonumber\\
 &= C \left\| P_{\mathcal{T}_{\mathbf{u}^n}} \mathcal{A}^* \mathcal{A} \left(\mathbf{u}-\mathbf{u}^n\right)\right\|^2  \leq C \left( 1+\delta_{3\mathbf{r}}\right)^2 \varepsilon \left\| \mathbf{u}- \mathbf{u}^n\right\|, \label{jed1}
\end{align}
where in the last estimate we  also used 
that $\left\| \mathbf{u}- \mathbf{u}^n \right\| \leq \varepsilon$. 

For the second term in \eqref{3terms} 
observe that
$(I - P_{\mathcal{T}_{\mathbf{u}^n}} )  ( \mathbf{u}   - \mathbf{u}^n ) = 
\mathbf{u} - P_{\mathcal{T}_{\mathbf{u}^n}}\mathbf{u} \in  \mathcal{M}_{\leq 3 \mathbf{r}}$ by Proposition~\ref{prop:Riemannian}.   
The TRIP implies therefore that the  spectrum of 
 $P_{\mathcal{T}_{\mathbf{u}^n} } ( I -   \mathcal{A}^* \mathcal{A})|_{W^n}  $ is  contained in the interval
 $ [- \delta_{3 \mathbf{r}} ,  \delta_{3 \mathbf{r}} ] $. 
With these observations and \eqref{eq:square} we obtain
\begin{eqnarray*}
  \| P_{\mathcal{T}_{\mathbf{u}^n} } (\mathbf{y}^{n+1} - \mathbf{u} )   \| 
      & = &    \|   P_{\mathcal{T}_{\mathbf{u}^n} }   \big(( \mathbf{u}  - \mathbf{u}^n  ) -  \mathcal{A}^* ( \mathbf{b} 
    - \mathcal{A} \mathbf{u}^n  
    ) \big) \|
\\
        & \leq  &    \|  P_{\mathcal{T}_{\mathbf{u}^n}  } \big( ( \mathbf{u}   - \mathbf{u}^n )  -  \mathcal{A}^* \mathcal{A}  \big( P_{\mathcal{T}_{\mathbf{u}^n}  } ( \mathbf{u}   - \mathbf{u}^n ) \big)   \big)\|  \\
        & + &  \|  P_{\mathcal{T}_{\mathbf{u}^n} }   \mathcal{A}^* \mathcal{A}  
        \big( (I - P_{\mathcal{T}_{\mathbf{u}^n} })  ( \mathbf{u}   - \mathbf{u}^n ) \big) \|
\\
        & \leq  &      \|   P_{\mathcal{T}_{\mathbf{u}^n} }  \big( ( \mathbf{u}   - \mathbf{u}^n )    -  
          \mathcal{A}^* \big( \mathcal{A} P_{\mathcal{T}_{\mathbf{u}^n}}   ( \mathbf{u}   - \mathbf{u}^n ) \big) \big) \| \\
&       +&  (1 + \delta_{3\mathbf{r}} )   \|  (I - P_{\mathcal{T}_{\mathbf{u}^n} }) ( \mathbf{u}  - \mathbf{u}^n)  \|   \ 
       \\
        & \leq  &      \|   P_{\mathcal{T}_{\mathbf{u}^n} } 
        \left(  I-\mathcal{A}^*  \mathcal{A} P_{\mathcal{T}_{\mathbf{u}^n} }\right)   ( \mathbf{u}   - \mathbf{u}^n )   \|\\
       &+ &  (1+\delta_{3\mathbf{r}}) \kappa \|   P_{\mathcal{T}_{\mathbf{u}^n} }( \mathbf{u}  - \mathbf{u}^n ) \|^2   \\
       & \leq & \delta_{3\mathbf{r}} \| P_{\mathcal{T}_{\mathbf{u}^n} } (\mathbf{u} - \mathbf{u}^n)\| 
       + (1+\delta_{3\mathbf{r}}) \kappa \|P_{\mathcal{T}_{\mathbf{u}^n} }(\mathbf{u} - \mathbf{u}^n)\|^2.
       \end{eqnarray*}
 Hence, for $\varepsilon$ sufficiently small, there exists a  factor  $ 0 <  \tilde{\rho}  < 1$ such that 
       \begin{equation} \label{jed2}
          \|  P_{\mathcal{T}_{\mathbf{u}^n} }  (\mathbf{y}^{ n+1}  - \mathbf{u}^n ) \|   
                \leq  \tilde{\rho}     \|  P_{\mathcal{T}_{\mathbf{u}^n} }  (\mathbf{u} - \mathbf{u}^n ) \|  \ . 
        \end{equation}
For the third term in \eqref{3terms}, first notice that by the Pythagorean theorem 
\begin{equation*}
\left\| \mathbf{u}^n-\mathbf{u} \right\|^2= \left\| \mathbf{u}^n- P_{\mathcal{T}_{\mathbf{u}^n}}\mathbf{u} \right\|^2 + \left\| P_{\mathcal{T}_{\mathbf{u}^n}}\mathbf{u}-\mathbf{u} \right\|^2.
\end{equation*}
Since $P_{\mathcal{T}_{\mathbf{u}^n}}\mathbf{u}^n=\mathbf{u}^n$, using \eqref{eq:equiv1} one obtains
\begin{equation} \label{jed3}
\left\|  P_{\mathcal{T}_{\mathbf{u}^n}}\mathbf{u} - \mathbf{u} \right\| \leq \sqrt{1-\gamma^{-2}} \left\| \mathbf{u}^n-\mathbf{u}\right\|.
\end{equation}
Combining the estimates \eqref{jed1}, \eqref{jed2} and \eqref{jed3} yields
\begin{equation*}
\left\| \mathbf{u}^{n+1} - \mathbf{u} \right\| \leq \rho \left\| \mathbf{u}-\mathbf{u}^n\right\|,
\end{equation*}
where $\rho=C\left(1+\delta_{3\mathbf{r}}\right)\varepsilon +  \tilde{\rho} +\sqrt{1-\gamma^{-2}} <1 $ for $\varepsilon$ and $\kappa$ small and $\gamma$ close enough to $1$. 
Consequently, 
 the sequence  $\mathbf{u}^n $ converges linearly to $\mathbf{u}$. 
\end{proof}

The weak point of the previous theorem is that this convergence can be guaranteed only in a very narrow neighborhood of the solution. 
To obtain global convergence, now for TIHT, we ask for an additional assumption on the iterates.

\begin{theorem} [Conditionally global convergence] \label{thm:conditional}
Let $\mathbf{b} =\mathcal{A}\mathbf{u}$ for $\mathbf{u} \in \mathcal{M}_{\leq\mathbf{r}}$ and let ${\mathbf{u}^n}$
be the iterates  \eqref{eq:yn}, \eqref{eq:un} of the tensor iterative hard thresholding algorithm, i.e.,
$\mathcal{R} (\mathbf{u}^{n+1} ) : = \mathbf{H}_{\mathbf{r}} \left(\mathbf{y}^{n+1}\right)$.
In addition, let's assume that $ \mathcal{A} $ satisfies the TRIP of order ${3 \mathbf{r}}$,    
i.e., $\delta_{3\mathbf{r}} \leq 1/2$. 
We further assume that the iterates satisfy, for all $n \in \mathbb{N}$,
\begin{equation}\label{condA}
\| \mathbf{u}^{n+1} -  \mathbf{y}^{n+1} \|  \leq \| \mathbf{u} -    \mathbf{y}^{n+1} \| .
\end{equation}
Then the sequence
$ \mathbf{u}^n \in \mathcal{M}_{\leq \mathbf{r}}$ converges linearly
to a unique solution $\mathbf{u}\in \mathcal{M}_{\leq \mathbf{r}}$ with rate 
$ \rho  < 1 $, i.e.,
$$
\| \mathbf{u}^{n+1}   - \mathbf{u} \| \leq  \rho \| \mathbf{u}^n - \mathbf{u}\| .
$$
\end{theorem} 
For details of the proof, we refer to \cite{rss2}. We note that the above result can be extended to robustness under noise on the measurements
and to tensors being only approximately of low rank, i.e., being close to ${\mathcal M}_{\mathbf{r}}$ but not 
necessarily on ${\mathcal M}_{\mathbf{r}}$.

Let us comment on the essential condition \eqref{condA}. 
In the case that ${\mathcal R}$ computes the best rank ${\mathbf{r}}$ approximation then this condition
holds since 
$$ \inf\{\| \mathbf{v} - \mathbf{y}^{n+1}\| : \mathbf{v} \in \mathcal{M}_{ \leq \mathbf{r}} \} \leq \| \mathbf{u} - \mathbf{y}^{n+1} \| $$ 
is trivially true. However, the best approximate is not numerically available and the truncated HSVD
only ensures the worst case error estimate
$$ \|  \mathbf{u}^{n+1} - \mathbf{y}^{n+1} \| \leq C(d)  \inf \{\| \mathbf{v} - \mathbf{y}^{n+1} \| : \mathbf{v} \in \mathcal{M}_{ \leq \mathbf{r}} \}. $$
Nevertheless, in practice this bound may be pessimistic for a generic tensor so that \eqref{condA} may be likely to hold.
In any case, the above theorem may at least {\it explain} why we observe recovery by TIHT in practice.

\subsection{Alternating least   squares  scheme {(ALS)}} 

An efficient and fairly simple method for computing (at least a local) minimizer of  
$\|\mathbf{u} - \mathbf{v}\|$ subject to 
$\mathbf{v} \in {\mathcal M}_{\mathbf{r}}$ is based on {\em alternating least squares  (ALS)}, which is 
a variant of block Gau\ss{}-Seidel optimization.  
In contrast to poor convergence experienced with the canonical format
 (CANDECOMP, PARAFAC) \cite{kolda}, 
ALS implemented appropriately  in the hierarchical formats
 has been observed  to be 
surprisingly powerful \cite{hrs-als}.  
Furthermore, and quite importantly,  it is robust with respect to over-fitting, i.e., allows optimization  in
the set  $\mathcal{M}_{\leq {\bf r}}$ \cite{hrs-als}. 
As a local optimization scheme, like the Riemannian optimization it converges only  to  a local minimum at best. 
This scheme applied to TT tensors is basically  a one-site DMRG (density matrix renormalization group)
 algorithm introduced for quantum spin systems in \cite{white,schollwoeck}.
 The basic idea  for computing $  \mathbf{u} = \tau ( \{ \mathbf{b}_{\alpha} : \alpha \in \mathbb{T} \}) $ by   the ALS or 
 block Gau\ss{}-Seidel method is  to compute the required components $ \mathbf{b}_{\alpha} $, 
one after each other.  
 Fixing the components  $ \mathbf{b}_{\alpha}  $, $ \alpha \in \mathbb{T} $, $ \alpha \not= t$, 
only the component $\mathbf{b}_t $  is left to be optimized in each iteration step. 
Before passing to the next iteration step,
the new iterate has to be transformed into the normal form $ \mathbf{b}_{t}^{n+1}$ by orthogonalization e.g.\ 
  by applying an SVD (without truncation)  
  or simply by  QR factorization. 

Let us assume that the indices $ \alpha \in \mathbb{T}$ are in a linear ordering $<$, which is consistent with the hierarchy. 
The  components given by the present iterate are denoted by  $ \mathbf{b}_{\alpha }^n $,  if $\alpha \geq  t$ respectively,  
$ \mathbf{b}_{\alpha}^{n+1}$ for $\alpha< t$. 
We optimize over $\mathbf{b}_t$ and  introduce a  corresponding 
 tensor by $$ \mathbf{u}^{n+1}_{-  t } := \tau \big( \{   \mathbf{b}_{\alpha}^{n+1} : \alpha <t \} 
 \cup \{ \mathbf{b}_t \}  \cup \{ \mathbf{b}_{\alpha}^n : \alpha > t\} \big)  \in \mathcal{H}_d \ . $$

Since the parametrization 
$ \tau \big( \{ \mathbf{b}_{\alpha} : {\alpha  \in \mathbb{T}} \}  \big) \in \mathcal{M}_{\leq \mathbf{r}}   $
 is multi-linear in its arguments  
$\mathbf{b}_{\alpha} $, the map  $  \tau_t^{n+1}   $ defined by 
$  \mathbf{b}_t \mapsto  \tau_t^{n+1}   (\mathbf{b}_t) :=   \mathbf{u}^{n+1}_{-t}  $
 is linear.  The  first order optimality condition for the present minimization is  
$$ 0 = \nabla  \mathcal{J} \circ \tau_t^n  ( \mathbf{b}_t ) =  
( \tau_t^n)^* \mathcal{A}^* \big( \mathcal{A} \,  \tau^n_t (\mathbf{b}_t)  - \mathbf{b} \big) \ , $$
which constitutes a linear equation for the unknown $ \mathbf{b}_t$. 
It is not hard to show the following facts.

\begin{theorem}
\begin{enumerate}
\item  Suppose that $ \mathcal{A} $ satisfies the TRIP of order $\mathbf{r}$ with TRIC
$ \delta_{\mathbf{r}} < 1  $ 
and that $\{\mathbf{b}_{\alpha}: \alpha \not=t \}$ is
orthogonalized as above. Then, since
$ \| \tau_t^n (\mathbf{b}_t ) \| = \| \mathbf{b}_t \| $, the TRIP
reads as
$$
( 1 - \delta_{\mathbf{r}} ) \| \mathbf{b}_t \|^2  \leq  \| \mathcal{A}
\tau_t^n ( \mathbf{b}_t ) \|^2 \leq ( 1 + \delta_{\mathbf{r}} ) \|
\mathbf{b}_t \|^2 \ .
$$
In addition,   the functional $ \mathcal{J} \circ \tau_t^n $ is strictly convex, and
$ \mathcal{J} \circ \tau_t^n $ possesses a unique minimizer  $ \mathbf{b}_t^n $.
\item For $ \mathbf{u}^n_{-t}:
 =  \tau^n_t ( \mathbf{b}^n_t )  $, the  sequence $ \mathbf{J} (  \mathbf{u}^n_{-t} ) $ is nonincreasing with $n$, and 
 it is decreasing 
 unless  $\mathbf{u}^n_{-t} $ is a stationary point, i.e.,
$ \nabla  \mathbf{J} ( \mathbf{u}^n_{-t} )  \perp \mathcal{T}_{\mathbf{u}^n_{-t}} $:
 In the latter case the algorithm stagnates. 
 \item  
The sequence of  iterates  $\mathbf{u}^n_{-t} $ is  uniformly bounded.   
\end{enumerate}
\end{theorem} 

This result implies at least the existence of a convergent subsequence.
However, no conclusion can be drawn whether this algorithm recovers the original low rank tensor ${\mathbf{u}}$ from 
${\mathbf{b}} = {\mathcal A} \mathbf{u}$.
For further convergence analysis of ALS, we refer e.g. to \cite{ru}. 

In \cite{wutao} 
 convergence of a Block Gau\ss{} Seidel method was shown by means of the Lojasiewicz-Kurtyka inequality. Also
nonnegative tensor completion has been discussed there.
 It is likely that these arguments apply also to the present setting. 
The ALS is simplified if one rearranges the tree in  each micro-iteration step such that one 
optimizes always the root. This can be easily done for TT tensors with left and right-orthogonalization
\cite{hrs-tt,hrs-als}, and can be modified for general  hierarchical  tensors as well. 
Often, it is preferable to proceed in an opposite order after the optimization of  all components (half-sweep). 
For the Gauss--Southwell variant, where one optimizes the  component with the largest 
defect, convergence estimates from gradient based methods can be applied \cite{su2}.
Although  the latter method converges faster, one faces a high computational overhead.  

Let us remark that  the Block Gau\ss{}-Seidel method and  ALS strategy can be used 
in various situations, in particular, as one ingredient in the TIHT and RGI algorithms from the previous section. 
For instance, ALS can be applied directly after a gradient step defining the operator $\mathcal{R}$ or
one can use a simple half-sweep for approximating  the gradient correction  $\mathbf{y}^{n+1} $
by  a rank $\mathbf{r}$ tensor in order to define the nonlinear projection $ \mathcal{R}$.

\section{Numerical results} 

For numerical tests, we concentrate on the HOSVD and the tensor iterative hardthresholding (TIHT) algorithm for recovering
order $d=3$ tensors from Gaussian measurement maps ${\mathcal A}:  \mathcal{H}_3=\bigotimes_{i=1}^3\mathbb{R}^{n_i} \to \mathbb{R}^m$, i.e.,
the entries of ${\mathcal A}$ identified with a tensor in $\mathbb{R}^{m} \otimes \bigotimes_{i=1}^3 \mathbb{R}^{ n_i}$ are i.i.d.\   
$\mathcal{N}\big( 0,\frac{1}{m}\big)$ random variables.

For these tests, we generate tensors  $\mathbf{u}  \in  \mathcal{H}_3 $
of rank $\mathbf{r} = (r_1, r_2, r_3)$  via its Tucker decomposition. Let us suppose that 
\begin{equation*}
 \mathbf{u}  (\mu_1 , \mu_2 , \mu_3 )  =  \sum_{k_1=1}^{r_1}\sum_{k_2=1}^{r_2}\sum_{k_3=1}^{r_3} \mathbf{c} ( k_1 , k_2 , k_3 ) 
 \mathbf{b}_{k_1}^1 (\mu_1)  \mathbf{b}_{k_2}^2 (\mu_2) \mathbf{b}_{k_3}^3 (\mu_3) 
\end{equation*}
 is the corresponding  Tucker decomposition.
Each entry of the core tensor $\mathbb{c}$ is taken independently
from the normal distribution, $\mathcal{N}  (0, 1)$, and the component tensors 
$ \mathbf{b}^j \in  \mathbb{R}^{n_j \times r_j  } $  are the first $r_j$  left singular vectors of a matrix
$\mathbf{M}^j  \in  \mathbb{R}^{n_j \times n_j}$ whose elements are also drawn independently
from the normal distribution $\mathcal{N}  (0, 1)$.

We then form the measurements ${\mathbf{b}} = {\mathcal A}\mathbf{u}$ and run the TIHT algorithm with the specified multi-linear 
rank $\mathbf{r} = (r_1, r_2, r_3)$ on ${\mathbf{b}}$. We test whether the algorithm successfully reconstructs the original tensor
and say that the algorithm converged if $\|\mathbf{u}-\hat{\mathbf{u}}\| <10^{-4}$. We stop the algorithm if it did not converge after $5000$ iterations.

Figures \ref{fig1}--\ref{fig3} present the recovery results for low rank tensors of size $10 \times 10 \times 10$ (Figures \ref{fig1} and Figure \ref{fig2}) and $6 \times 10 \times 15$ (Figure \ref{fig3}). The horizontal axis
represents the number of measurements taken with respect to
the number of degrees of freedom of an arbitrary tensor of this
size. To be more precise, for a tensor of size  $n_1 \times n_2 \times n_3$, the
number $\overline{n}$  on the horizontal axis represents 
 $m  =\lceil n_1 n_2 n_3 \frac{\overline{n}}{100}\rceil$  measurements. The vertical axis represents the percentage of the successful recovery.  For fixed tensor dimensions $n_1 \times n_2 \times n_3$, fixed HOSVD-rank $\mathbf{r}=\left(r_1,r_2,r_3\right)$ and fixed number of measurements $m$, we performed $200$ simulations.

\begin{figure}[!ht]
 \begin{minipage}{0.7\textwidth}
  \centerline{\includegraphics[width=0.7\textwidth]{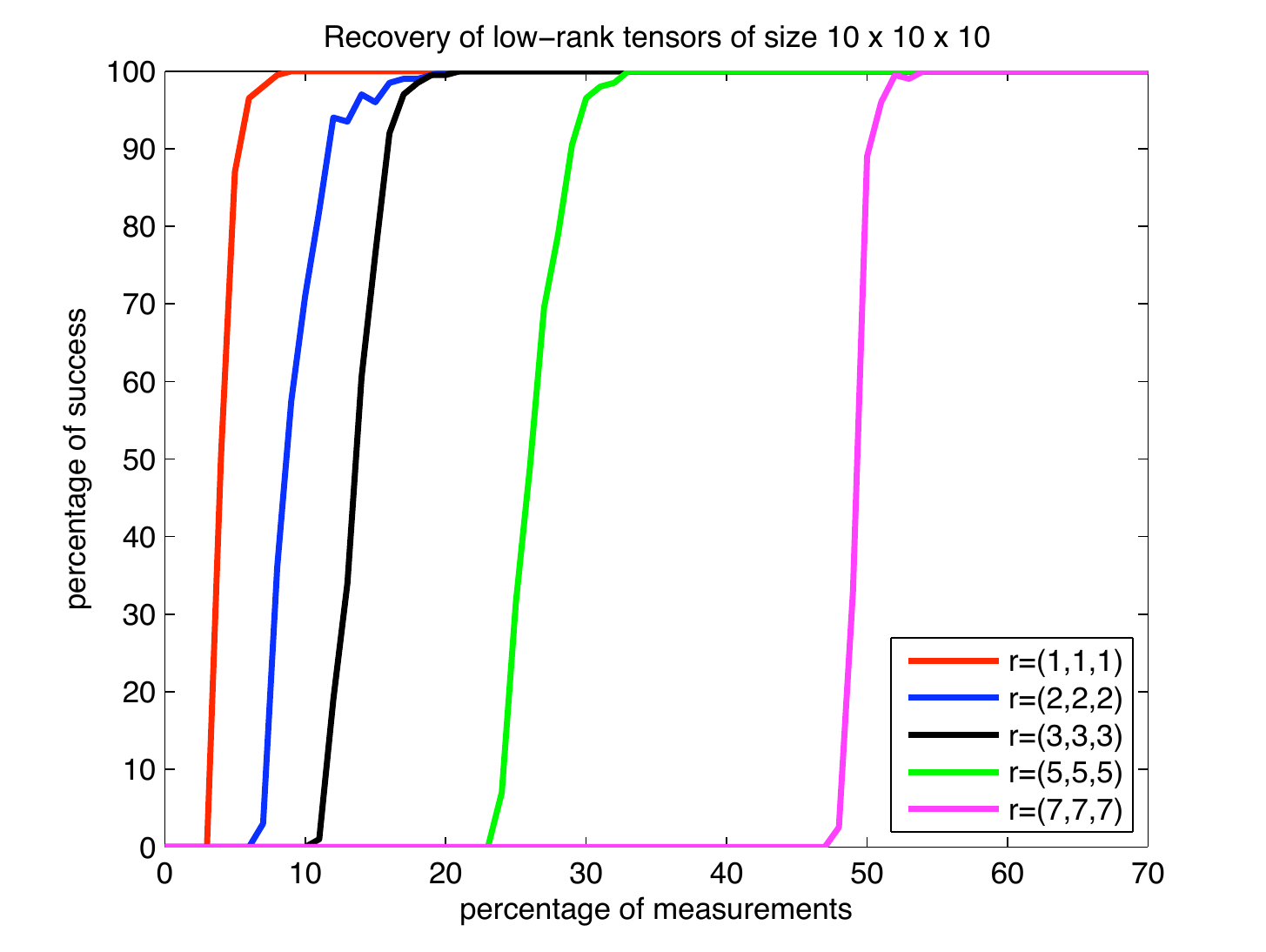}}
  \caption{Numerical results for \mbox{$10 \times 10 \times 10$} tensors with same $k$-ranks.}\label{fig1}
 \end{minipage}
\hfill
 \begin{minipage}{0.7\textwidth}
  \centerline{\includegraphics[width=0.7\textwidth]{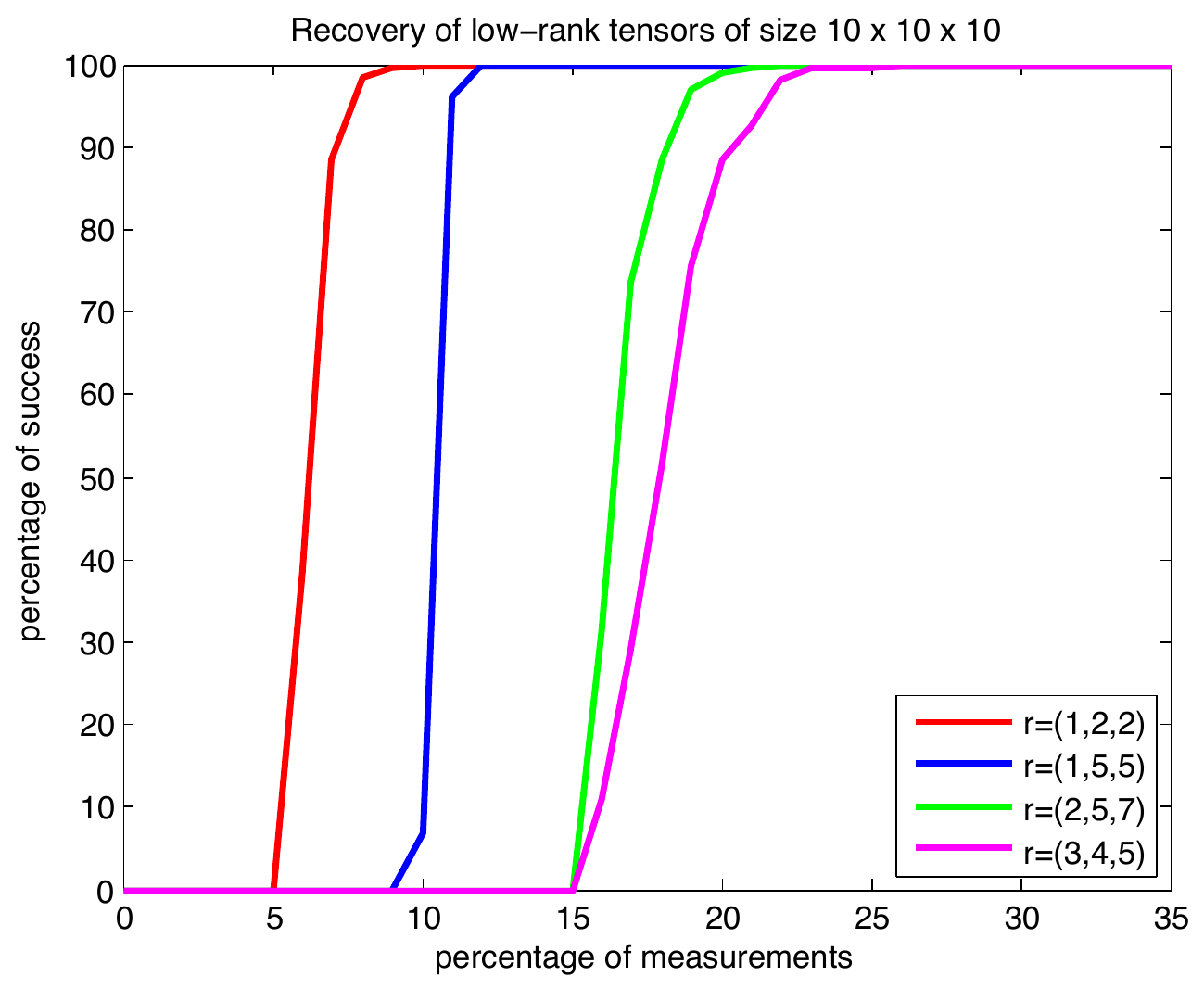}}
  \caption{Numerical results for \mbox{$10 \times 10 \times 10$} tensors with \mbox{different} $k$-ranks.}\label{fig2}
 \end{minipage}
 \hfill
  \begin{minipage}{0.7\textwidth}
  \centerline{\includegraphics[width=0.7\textwidth]{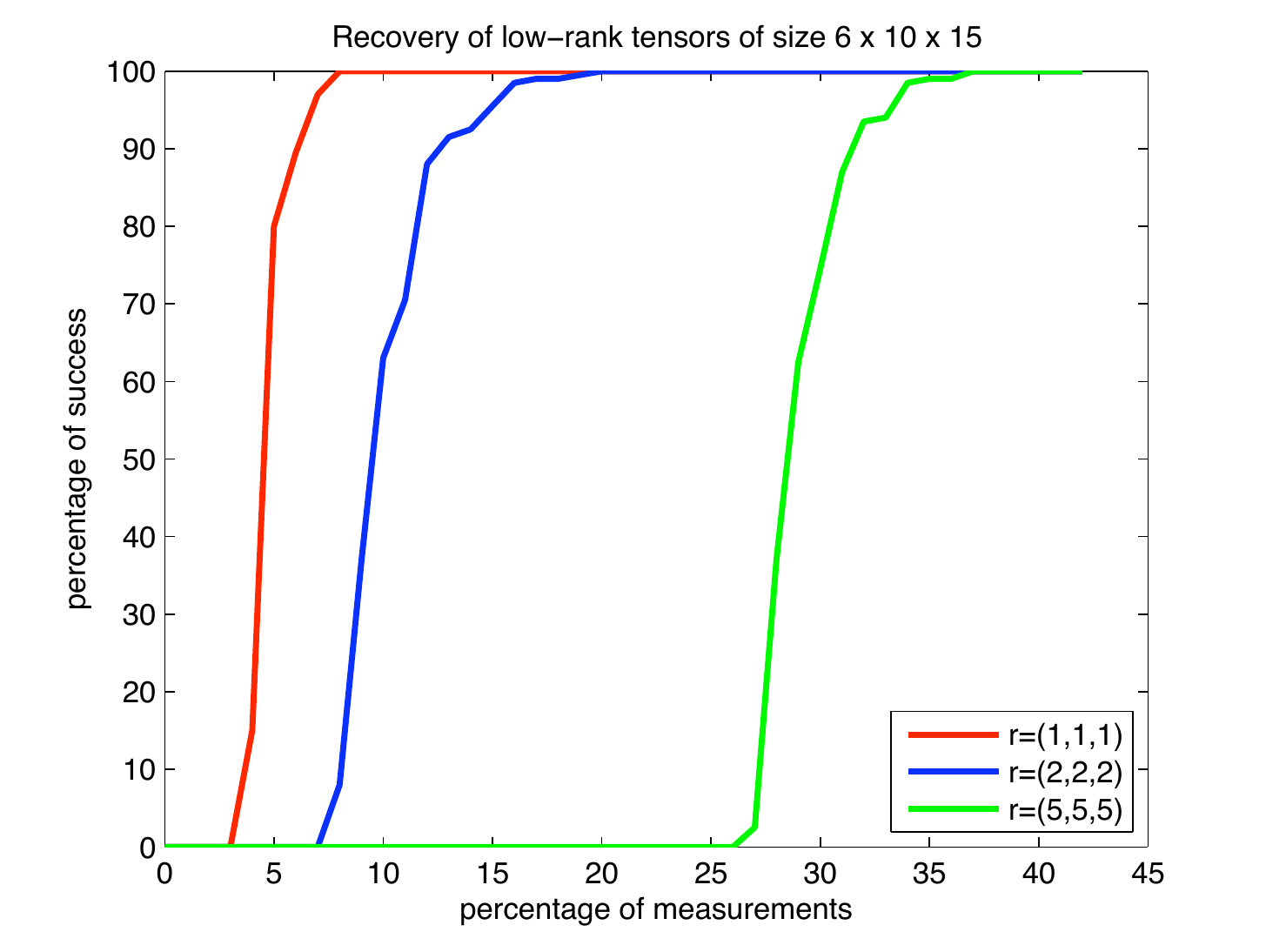}}
  \caption{Numerical results for \mbox{$6 \times 10 \times 15$} tensors with same $k$-ranks.}\label{fig3}
 \end{minipage}
 \caption[justification=centering,margin=2cm]{Numerical results for tensor IHT algorithm.}
\label{Figures}
\end{figure}

Table \ref{tableNum} complements Figure \ref{Figures}. With $\%_{\max}$ we denote the maximal percentage of measurements for which we did not manage to recover even one tensor out of $200$. 
The minimal percentage of measurements for full recovery is denoted by $\%_{\min}$. The last column represents the number of iterations needed for full recovery with $m=\lceil n_1 n_2 n_3 \frac{\%_{\min}}{100} \rceil$ number of measurements.

\begin{table}[h]
\centering
\caption{Numerical results for tensor IHT algorithm.}
\begin{tabular}{| c| c | c | c | c|}
\hline
$\,n_1 \times n_2 \times n_3\,$ & rank          &	  $\%_{\max}$ & $\%_{\min}$ & \# of iterations for $\%_{\min}$ \\
\hline
$10 \times 10 \times 10\,$ & $\,(1,1,1)\,$  & 	$3$		&	$9$		&	$321$	\\
$10 \times 10 \times 10$ & $(2,2,2)$  & 	$6$		&	$20$	&	$185$	\\
$10 \times 10 \times 10$ & $(3,3,3)$  & 	$10$	&	$21$	&	$337$	\\
$10 \times 10 \times 10$ & $(5,5,5)$  & 	$23$	&	$33$	&	$547$	\\
$10 \times 10 \times 10$ & $(7,7,7)$  & 	$47$	&	$54$	&	$1107$	\\
\hline \hline
$10 \times 10 \times 10$ & $(1,2,2)$  & 	$5$		&	$10$	&	$588$	\\
$10 \times 10 \times 10$ & $(1,5,5)$  & 	$9$		&	$12$	&	$1912$	\\
$10 \times 10 \times 10$ & $(2,5,7)$  & 	$15$	&	$22$	&	$696$	\\
$10 \times 10 \times 10$ & $(3,4,5)$  & 	$15$	&	$26$	&	$384$	\\
\hline \hline
$6 \times 10 \times 15$ & $(1,1,1)$  & 		$3$		&	$8$		&	$511$	\\
$6 \times 10 \times 15$ & $(2,2,2)$  & 		$7$		&	$20$	&	$214$	\\
$6 \times 10 \times 15$ & $(5,5,5)$  & 		$26$	&	$37$	&	$501$	\\
\hline
\end{tabular}
\label{tableNum}
\end{table}

\section{Concluding remarks}    
In this chapter we considered low rank tensor recovery for hierarchical tensors  extending the classical Tucker format to a multi-level framework. 
For low ranks, this model can  break  the curse of dimensionality. Its number of degrees of freedom scale 
like $\mathcal{O} (ndr + d r^3)  \ll \mathcal{O} (n^d)$ 
and $ \mathcal{O} (ndr^2 )$ for TT tensors instead of $\mathcal{O} (n^d)$.    
Under the assumption of a tensor restricted isometry property, we have shown local convergence for Riemannian gradient iterations and
global convergence under a certain condition on the iterates of the tensor iterative hard thresholding algorithm  
for hierarchical tensors,  
including the classical Tucker format as well as tensor trains. 
For instance for TT tensors, an estimate of the TRIP for Gaussian measurement maps
was provided that requires the number of measurements to scale like $m \sim ndr^2 \log (dr)$.
However, it is still not clear whether the logarithmic factor is needed.

Let us finally mention some open problems. 
One important task is to establish global convergence to the original tensor of any of the discussed algorithms, without additional assumptions such as \eqref{condA} on the iterates.
In addition, robustness and stability for the Riemannian gradient method are still open.
Further, also the TRIP related to general HT tensors for Gaussian measurement maps is not yet established.  
Since the TRIP does not hold for the completion problem, it is not clear yet whether a low rank tensor 
can be recovered from less than $ \mathcal{O} (n^{d/2} ) $ entries. 
 
%
%
%

\end{document}